\documentclass[11pt,reqno]{amsart}
\usepackage[utf8x]{inputenc}
\usepackage{ucs}
\usepackage{amsmath,slashed}
\usepackage{amsfonts}
\usepackage{mathtools}
\usepackage{amssymb}
\usepackage[shortlabels]{enumitem}
\usepackage{bbm}
\usepackage{graphicx}
\usepackage[left=2cm,right=2cm,top=2cm,bottom=2cm]{geometry}

\def\NN{\mathbb{N}}
\def\RR{\mathbb{R}}

\def\ZZ{\mathbb{Z}}

\def\TT{\mathbb{T}}

\DeclareMathOperator\Erf{Erf}

\newcommand\cM{\mathcal{M}}

\newcommand\tH{\widetilde{H}}

\newcommand\tka{\tilde\kappa}

\newcommand\chis{\chi_{\mathrm{sm}}}
\newcommand\chim{\chi_{\mathrm{med}}}
\newcommand\chil{\chi_{\mathrm{lar}}}

\newcommand\crit{N(\nabla u,R)}

\newlength{\intwidth}

  \newcommand\uno{\mathrm{I}}
  \newcommand\dos{\mathrm{II}}
  \newcommand\tres{\mathrm{III}}

\usepackage{amssymb}
\usepackage[centertags]{amsmath}
\usepackage{verbatim}

\newcommand{\pd}{\partial}
\newcommand\om\omega
\newcommand{\al}{\alpha}
\newcommand\ze\zeta
\newcommand{\be}{\beta}
\newcommand{\si}{\sigma}
\newcommand{\Si}{\Sigma}
\newcommand{\De}{\Delta}
\def\RR{\mathbb{R}}

\DeclareMathOperator\Real{Re}
\DeclareMathOperator\Imag{Im}

\def\cR{\mathcal R}

\def\TT{\mathbb{T}}
\def\S{\mathbb{S}}

\newcommand\vp\varphi
\newcommand\vr\varrho
\newcommand\tvr{\tilde{\varrho}}
\newcommand\ka\kappa
\newcommand\te\theta
\newcommand\vka\varkappa

\newcommand\tI{\widetilde{I}}

\newcommand\de\delta
\newcommand\La\Lambda
\newcommand\la\lambda
\newcommand\ga\gamma

\newcommand\Ga\Gamma

\newcommand\cI{\mathcal I}
\newcommand\cJ{\mathcal J}
\newcommand\R{_{\mathrm{R}}}
\newcommand\I{_{\mathrm{I}}}

\newcommand{\triple}[1]{{\left\vert\kern-0.25ex\left\vert\kern-0.25ex\left\vert #1
        \right\vert\kern-0.25ex\right\vert\kern-0.25ex\right\vert}}

\usepackage[toc, page]{appendix}

\newcommand{\bR}{\mathbb{R}}

\newcommand{\bE}{\mathbb{E}}
\newcommand{\bP}{\mathbb{P}}
\newcommand{\bZ}{\mathbb{Z}}

\DeclareMathOperator\var{Var}
\DeclareMathOperator\cov{Cov}

\makeatletter
\renewcommand*{\@fnsymbol}[1]{\ensuremath{\ifcase#1\or *\or \star\or ***\or
   \mathsection\or \mathparagraph\or \|\or **\or \dagger\dagger
   \or \ddagger\ddagger \else\@ctrerr\fi}}
\makeatother

\usepackage{fancyhdr}
\newcommand\ep\varepsilon

\newtheorem{theorem}{Theorem}[section]
\newtheorem{lemma}[theorem]{Lemma}

\newtheorem{proposition}[theorem]{Proposition}
\newtheorem{corollary}[theorem]{Corollary}

\theoremstyle{definition}
\newtheorem{remark}[theorem]{Remark}

\numberwithin{equation}{section}

\renewcommand\le\leqslant
\renewcommand\leq\leqslant
\renewcommand\geq\geqslant
\renewcommand\ge\geqslant

\usepackage[pdftex]{hyperref}
\hypersetup{
	colorlinks=true,
	linkcolor=blue,
	citecolor=blue,
	filecolor=blue,
	urlcolor=blue,
}

\title[Critical point asymptotics]{Critical point asymptotics for Gaussian random waves with
  densities of any Sobolev regularity}

\author{Alberto Enciso}
\address{Instituto de Ciencias Matem\'aticas, Consejo Superior de
  Investigaciones Cient\'\i ficas, 28049 Madrid, Spain}
\email{aenciso@icmat.es}

\author{Daniel Peralta-Salas}
\address{Instituto de Ciencias Matem\'aticas, Consejo Superior de
 Investigaciones Cient\'\i ficas, 28049 Madrid, Spain}
\email{dperalta@icmat.es}

\author{\'Alvaro Romaniega}
\address{Instituto de Ciencias Matem\'aticas, Consejo Superior de
 Investigaciones Cient\'\i ficas, 28049 Madrid, Spain}
\email{alvaro.romaniega@icmat.es}

\begin{document}
\maketitle

\begin{abstract}
  We consider Gaussian random monochromatic waves~$u$ on the plane
  depending on a real parameter~$s$ that is directly related to the
  regularity of its Fourier transform. Specifically, the Fourier
  transform of~$u$ is $f\,d\si$, where $d\si$ is the Hausdorff measure
  on the unit circle and the density $f$ is a function on the circle
  that, roughly speaking, has exactly $s-\frac12$ derivatives in~$L^2$
  almost surely. When $s=0$, one recovers the classical setting for
  random waves with a translation-invariant covariance-kernel.  The
  main thrust of this paper is to explore the connection between the
  regularity parameter~$s$ and the asymptotic behavior of the number
  $\crit$ of critical points that are contained in the disk of
  radius~$R\gg1$. More precisely, we show that the expectation~$\bE\crit$ grows like the area of the disk when the regularity is
  low enough ($s<\frac32$) and like the diameter when the regularity
  is high enough ($s>\frac52$), and that the corresponding exponent
  changes according to a linear interpolation law in the intermediate
  regime. The transitions occurring at the endpoint cases involve the
  square root of the logarithm of the radius. Interestingly, the
  highest asymptotic growth rate occurs only in the classical
  translation-invariant setting, $s=0$. A key step of the proof of
  this result is the obtention of precise asymptotic expansions for
  certain Neumann series of Bessel functions. When the
  regularity parameter is $s>5$, we show that in fact $\crit$ grows
  like the diameter with probability~1, albeit the ratio is not a
  universal constant but a random variable.
\end{abstract}

\section{Introduction}

Nazarov and Sodin have developed some powerful techniques to derive
asymptotic laws for the distribution of the zero set of smooth
Gaussian functions of several
variables~\cite{NS09,NS16}. Specifically, their theory applies to two
different but related settings: the restriction to large balls of
Gaussian functions on Euclidean space with translation-invariant
covariance kernels and to Gaussian ensembles of high degree
polynomials on the sphere or the torus with asymptotically
translation-invariant kernels. In the first
setting, a prime example arising in spectral theory is the study of
Gaussian random monochromatic waves; in the second, that of random
spherical harmonics of high frequency.

In this paper we are concerned with asymptotic laws for the
number of critical points (i.e., the zeros of the gradient). We
consider this question in the context of Gaussian random
monochromatic waves on the plane, which are solutions to the Helmholtz
equation on~$\RR^2$,
\begin{equation}\label{Helmholtz}
\De u + u = 0\,.
\end{equation}
As is well known, the study of
critical points is a central topic in spectral theory~\cite{Ya82,Ya93,JN,Bu20}
(and, in general, in the geometric study of solutions to
differential equations~\cite{Wa50,Al87,Mag1,EP18}), both in the deterministic and
random settings. This is partly because
they are very closely related to the geometry of the nodal components.

When $u$ is polynomially bounded, the Helmholtz equation
simply means that~$u$ is the Fourier transform of a distribution
supported on the unit circle, which we identify with
$\TT:=\RR/2\pi\ZZ$ via the map
\begin{equation}\label{defE}
E(\phi):=(\cos\phi,\sin\phi)\,.
\end{equation}
Solutions to the Helmholtz
equation are necessarily analytic, but their Fourier transforms do not have
any a priori regularity properties. There are some connections,
though, between the regularity of the Fourier transform
and the decay rate of~$u$ at infinity.
Most important is the classical result of Herglotz ensuring
that $u$ has the sharp fall-off at
infinity (which is as $|x|^{-\frac12}$ in a space-averaged sense) if and only
if one can write
\begin{equation}\label{uf}
u(x)=\int_{\TT}e^{-i x\cdot E(\phi)}\, f(\phi)\, d\phi
\end{equation}
with some square-integrable density~$f$, and that in this case the norm
$\|f\|_{L^2(\TT)}$ quantitatively captures the decay rate
of~$u$. For details and generalizations, see e.g.~\cite[Appendix A]{EPR20}.

The main thrust of this paper is to understand the connection between
the distribution of the critical points of~$u$, defined as
in~\eqref{uf}, and the regularity of the density~$f$. To this end, we
consider the usual ansatz for random plane waves~\cite{CS19,SW19} and tweak it by
introducing a real parameter $s\in\RR$ to control the
regularity of~$f$:
\begin{equation}\label{defu}
u(x) := \sum_{l\neq 0} a_l \,|l|^{-s} e^{i l\te}\, J_l(r)\,.
\end{equation}
Here the real and imaginary part of $a_l$
are independent standard Gaussian random variables subject to the
constraint $a_l=(-1)^l\overline{a_{-l}}$ (which makes~$u$ real valued),
$(r,\te)\in \RR^+\times\TT$ are the polar coordinates. This is equivalent to
taking the Gaussian random density
\begin{equation}\label{deff}
f(\phi):= \frac1{2\pi}\sum_{l\neq 0} i^l a_l |l|^{-s} e^{i l\phi}
\end{equation}
and then defining~$u$ through the formula~\eqref{uf}, which must be understood in
the sense of distributions.

Of course, the rationale behind this definition is that
$\{|l|^{-s} e^{il\phi}\}_{l\neq 0}$ is an
orthonormal basis of the Sobolev space $\dot H^s(\TT)$ of functions
with zero mean and~$s$ derivatives in~$L^2$, which reduces to the
space of square-integrable functions of zero mean when $s=0$. The covariance kernel
of~$u$ is translation-invariant when $s=0$, so the Nazarov--Sodin
theory is applicable in this case (see Remark \ref{KernelNS} for details), but this is not the case for nonzero~$s$. One should note that the proofs work verbatim if one replaces the weight $|l|^{-s}$ by a more general expression such as
\begin{equation}\label{defsil}
\si_{l}=\si_{-l}=|l|^{-s} +p_{-s-1}(l)\,,
\end{equation}
where the function $p_{-s-1}(t)$ is an arbitrary classical symbol of order $-s-1$
(which does not necessarily vanish at~0). The resulting constants, however, depend on the specific sequence $\si_l$.

It is not hard to see that the parameter~$s$
describes the regularity of the density in the sense that~$f$ has exactly
$s-\frac12$ derivatives in~$L^2$ almost surely, as measured using
Sobolev or Besov spaces. Specifically, one can show that, for any $\de>0$,\[
  f\in \Big[H^{s-\frac12-\de}(\TT)\backslash H^{s-\frac12}(\TT)\Big] \cap
  \Big[B^{s-\frac12}_{2,\infty}(\TT)\backslash B^{s-\frac12+\de}_{2,\infty}(\TT)\Big]
\]
with probability~1; see Proposition~\ref{P.reg} for details.

Our main result provides an asymptotic estimate for the growth of
the expected number of critical points contained in a disk of
large radius~$R$, which we denote by
\[
  \crit := \#\{x\in B_R: \nabla u(x)=0\}\,,
\]
as a function of the
regularity parameter~$s$. It is
elementary that this quantity is an upper bound for the expected number of
nodal components contained in~$B_R$.
With the usual ansatz for random
plane waves, it is well known that
$\crit$ grows asymptotically like the area of the disk; more
precisely~\cite{Be20}, when $s=0$ one has
\[
  \bE\crit \sim \ka(0)\,R^2\,,
\]
where $\ka(0):= 1/(2 \sqrt3)$ and where the notation $q(R)\sim Q(R)$
means that the quotient ${q(R)}/{Q(R)}$ tends to~1 as $R\to\infty$.

We should mention from the onset that the effect of changing the
regularity parameter~$s$ can be quite drastic, as one should not
expect that the number of critical points grows like the area in all
regularity regimes. To illustrate this, recall that, when $s=0$, the
Nazarov--Sodin theory ensures the number of nodal components of~$u$
contained in~$B_R$ grows as
\[
N(u,R)\sim \nu_0R^2
\]
almost surely for some constant $\nu_0>0$. In contrast, the results proven in~\cite{EPR20}
show that
\[
N(u,R)\sim\nu_\infty R
\]
almost surely for $s>4$, with $\nu_\infty:=1/\pi$. Understanding the asymptotic
behavior of the number of nodal components in other regimes is an
extremely challenging open problem. Consequently, our main objective
in this paper is to analyze the intriguing transitions between distinct asymptotic
regimes in the simpler case of critical points.

In the case of critical points, it is also natural to wonder about the
asymptotic growth in the case of very negative
regularities $s<0$. Recall that, by the Faber--Krahn inequality, the
number of nodal components of a solution to the Helmholtz equation
contained in~$B_R$ is at most $cR^2$, where $c$ is a universal
constant. However, the number of critical points is not bounded a
priori: in Appendix~\ref{A.manycp} we show that, given any continuous
function $\rho:\RR^+\to\RR^+$, there exists a solution to the
Helmholtz equation on~$\RR^2$ having at least $\rho(R)$ nondegenerate
critical points in~$B_R$, for all $R>1$. Thus, one could in principle
expect the average number of critical points in a large ball~$R$ to
have a fast growth in~$R$ for small enough regularities.

Our main result provides a satisfactory, and quite surprising, answer
to both questions. It turns out that the growth of the expected number
of critical points is like the square of the radius for $s<\frac32$,
linear for $s>\frac52$, and the corresponding exponent changes according
to a linear interpolation law in the intermediate regime
$\frac32<s<\frac52$. The transitions occurring at the endpoint cases
involve not only a power law, but also the square root of the
logarithm of the radius. Furthermore, the highest asymptotic
growth of the expected number of critical points is attained exactly
for $s=0$, that is, in the usual setting of random plane waves.

\begin{theorem}\label{T.main} 
  For any real~$s$, the following statements hold:
  \begin{enumerate}
\item There exist explicit positive constants $\ka(s) ,\tka_{\frac32},\tka_{\frac52}$ such that the expected number of critical points of the Gaussian random
function~$u$ satisfies
\begin{align*}
  \bE {\crit}\sim
  \begin{cases}
    \ka(s) \,R^2 &\text{if }\quad s<\frac32\,,\\[2mm]
\tka_{\frac32} \frac{R^2}{ \sqrt{\log R}} &\text{if }\quad
                                                     s=\frac32\,,\\[2mm]
\ka(s) \,R^{2-(s-\frac32)} & \text{if }\quad \frac32<s<\frac52\,,\\[2mm]
\tka_{\frac52} R \sqrt{\log R} &\text{if }\quad
                                                     s=\frac52\,,\\[2mm]
                                                                   {\ka(s)
                                                                   }\,R
                                 &\text{if }\quad s>\frac52\,.
\end{cases}
\end{align*}
\item In the region where the growth of
$\bE\crit$ is volumetric, the constant $\ka(s) $ depends continuously
on~$s$. More precisely, $\ka(s)$ is a $C^\infty$ function of
$s\in (-\infty,\frac12)\cup(\frac12,\frac32]$ but it is only Lipschitz
at $s=\frac12$. Furthermore, $\ka(s)$ is strictly increasing on~$(-\infty,0)$, strictly decreasing on
$(0,\frac32)$, and tends to~$0$ as $s\to-\infty$ and as $ s\to
\frac32^-$. In the region $s\in(\frac32,\frac52)\cup (\frac52,\infty)$ the constant $\ka(s)$ is also $C^\infty$.
\end{enumerate}
\end{theorem}
 \begin{figure}[h]
	\centering
	\includegraphics[width=.9\linewidth]{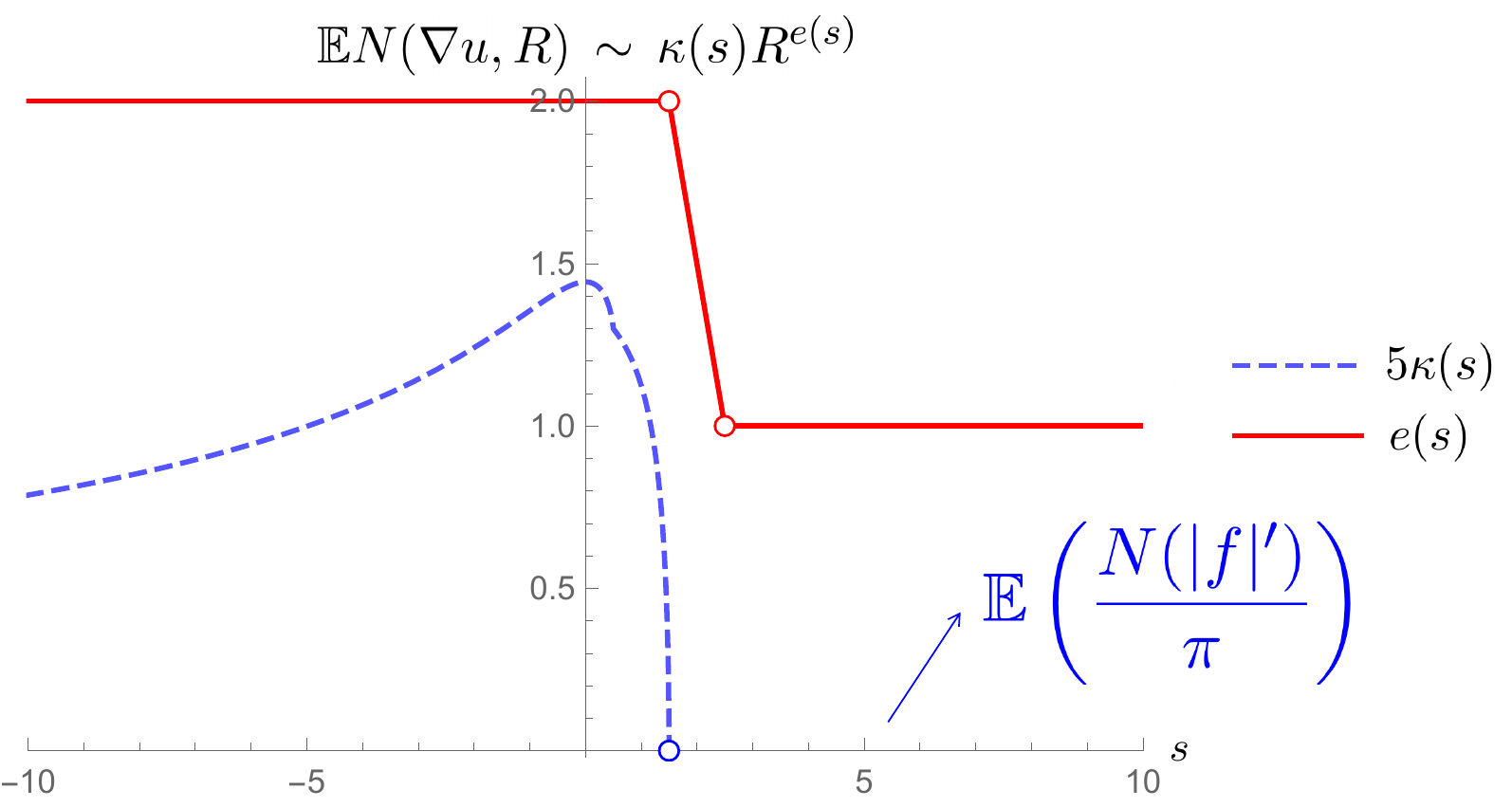}
	\caption{Consider the asymptotic behavior of $\bE\crit \sim
          \kappa(s)R^{e(s)}$ proved in Theorem~\ref{T.main}. 
          In red, we have plotted the exponent $e(s)$ as a function of
          $s\in\RR\backslash\{\frac32,\frac52\}$. Logarithmic effects appear at the endpoints $s=3/2$   and  $s=5/2$. In blue, we have
          plotted $\ka(s)$ in the region where the asymptotic growth
          is volumetric, $s<\frac12$. The maximum of
          $\ka(s)$ in this region is attained at $s=0$ and that
          $\ka(s)$ is not continuously differentiable at
          $s=1/2$. The reader can find a plot of $\kappa(s)$ in the
          range $s\in(\frac32,\frac52)$ in Figure~\ref{F.2}, cf.\
          Section~\ref{S.main}. Note that $\kappa(s)={\bE N(|f|')/\pi
          }$ by Theorem~\ref{T.as}. 
}
	\label{F.3}
\end{figure}
Figure~\ref{F.3} summarizes Theorem~\ref{T.main} in a more visual way. The fact that the highest asymptotic growth for the number of critical
points occurs precisely in the translation-invariant case $s=0$ is
somewhat surprising. Naively one could expect that rougher density
functions, which feature wilder oscillations, would exhibit more
critical points. Theorem~\ref{T.main} shows that, strictly speaking,
this is only the case for regularities $s>0$.

Let us now discuss the proof of Theorem~\ref{T.main}. The asymptotic analysis
of $\crit$ hinges on the celebrated Kac--Rice counting formula, which,
under suitable technical hypotheses,
expresses the expected number of zeros of a random field (in this
case, the gradient $\nabla u$) has in terms of a multivariate
integral. As is well known, this formula has been used profusely in the literature~\cite{EF16,NS16,Be20,BMW19}, and
in particular lies at the heart of the computation of $\bE\crit$ for
$s=0$ and of the finer asymptotics bounds for the expected number of extrema and
saddle points and for higher order correlations obtained
in~\cite{Be20} also in the translation-invariant case $s=0$.

The coefficients that appear in the Kac--Rice integral formula
involve, via the variance matrix of $\nabla u$, weighted series of
Bessel functions of the form
\begin{equation}\label{cJ}
\cJ_{s,m,m'}(r):= \sum_{l=1}^\infty l^{-2s} J_{l+m}(r)\, J_{l+m'}(r)\,,
\end{equation}
where $m$ and $m'$ are certain integers. $\cJ_{s,m,m'}$ is sometimes called in the literature a second type Neumann series. It is clear that the way each
term $J_{l+m}(r)\, J_{l+m'}(r)$ contributes to the sum for $r\gg1$ and
$l\gg1$ will depend on whether the ``angular frequency'' $l$ is much
larger than~$r$, much smaller than~$r$, or roughly of the same size;
moreover, the effect of each group of angular frequencies will have a
different relative weight in the sum depending on the power~$s$
appearing in $l^{-2s}$. More precisely, a key step of the proof is
to establish the following technical result, which controls the
asymptotic behavior of~$\cJ_{s,m,m'}(r)$:

\begin{lemma}\label{L.Bessel}
  For any pair of nonnegative integers $m,m'$ and any real~$s$, the large-$r$
  asymptotic behavior of $\cJ_{s,m,m'}$ is
  \begin{align*}
\cJ_{s,m,m'}(r)&= c^1_{s,m-m'}\, r^{-2s} + o(r^{-2s}) &\text{if}\quad &
                                                                s<\tfrac12\,,\\
    \cJ_{s,m,m'}(r)&= c^2_{m-m'}\frac{\log r}r +
                     O(r^{-1}) &\text{if}\quad &
                                                                s=\tfrac12
                                                                                              \text{
                                                                                              and
                                                                                              $m-m'$
                                                                                              is
                                                 even}\,,\\
    \cJ_{s,m,m'}(r)&= \frac{c^3_{m-m'}-c^4\sin(2r-c^7_{m+m'})}r +
                     o(r^{-1}) &\text{if}\quad &
                                                                s=\tfrac12
                                                                                              \text{
                                                                                              and
                                                                                              $m-m'$
                                                                                              is
                                                 odd}\,,\\
    \cJ_{s,m,m'}(r)&= \frac{c^5_{s,m-m'} -c^6_{s}\sin(2r-c^7_{m+m'})}r +
                     o(r^{-1}) &\text{if}\quad &
                                                                s>\tfrac12
  \end{align*}
  with some explicit constants that will be defined later on.
               \end{lemma}

Ultimately, the different asymptotic regimes that the expectation of
$\crit$ can exhibit can be traced back to the asymptotic behavior of
functions of the form~\eqref{cJ}. One should note that, in general, the highly
oscillatory nature of summands in~\eqref{cJ} makes the analysis of the
asymptotic behavior of $\cJ_{s,m,m'}(r)$ rather subtle. An
exception to this general fact is precisely the case $s=0$, where all the
associated series can be computed exactly using that the covariance
kernel of~$u$ is translation-invariant (or, equivalently, the addition formula
for Bessel functions); this makes it much easier to analyze the corresponding
asymptotic behavior of $\bE \crit$. To illustrate this fact, in the very short
Appendix~\ref{ApCompTI} we carry out the analysis of the translation
invariant case $s=0$.

In the particular case of smooth enough density functions,
one can use the methods of our previous paper~\cite{EPR20} to
understand the asymptotic behavior of the number of critical points
(not only of its expectation value) in greater detail. Specifically, one can prove the following:
\begin{theorem}\label{T.as}
  If $s>5$,
  \[
\crit \sim {\frac{N(|f|')}{\pi} R}
\]
with probability~$1$.  In particular, $\crit$ grows linearly almost surely.
\end{theorem}

Here the random variable $N(|f|'):=\#\{\phi\in\TT:
|f(\phi)|'=0\}$ (which is at least~$2$ almost surely) denotes
the number of critical points of the (non-Gaussian) random function~$|f|$. In
particular, the asymptotic growth of $\crit$ is linear with probability~1,
albeit the ratio is not a universal constant but a random variable. In view of
Theorem~\ref{T.main}, a
consequence of this asymptotic formula is an explicit formula for the
expectation $\bE N(|f|')$ when $s>5$.

The paper is organized as follows. In Section~\ref{S.regularity}, we
start by showing the relation between the parameter~$s$ and the
regularity of the random function~$u$. Sections~\ref{S.Bessel},
\ref{S.main} and~\ref{ap:smooth case} are respectively devoted to the
proofs of Lemma~\ref{L.Bessel} and Theorems~\ref{T.main}
and~\ref{T.as}. We have divided each of these sections into a number
of subsections to emphasize the main ideas of each proof. The paper
concludes with two Appendices. In Appendix~\ref{A.manycp}, we construct solutions to the
Helmholtz equation on the plane for which the number of nondegenerate
critical points contained in~$B_R$ grows arbitrarily fast as
$R\to\infty$. In Appendix~\ref{ApCompTI}, we revisit the
translation-invariant case ($s=0$) and explain the key simplifications
that appear in this extremely important case.

\section{Almost sure regularity of the random density function}
\label{S.regularity}

Our objective in the section is to show that, with probability~1, the
Gaussian random function~$f$, defined in~\eqref{deff}, has exactly
$s-\frac12$ derivatives in~$L^2$, measured using suitable Sobolev or
Besov spaces.

To prove the main result we will need the following version of the
strong law of large numbers for sequences of random variables that are labeled by two integers:

\begin{lemma}\label{lem:double SLLN} Let $\{K_N\}_{N=1}^\infty$ be a
  sequence of positive integers such that 
  \[
\liminf_{M\to\infty}\frac{K_M}{\sum_{N=1}^MK_N}>0\,.
        \]
    If $\{b_{N,k}: 1\leq k\leq K_N,\: N\geq1\}$ are i.i.d. ~random
variables with mean $\mu$, then 
	\begin{equation*}
		\lim_{N\to\infty}\frac{1}{K_N}\sum_{k=1}^{K_N}b_{N,k}=\mu
	\end{equation*}
almost surely.
\end{lemma}
\begin{proof}
	The strong law of large numbers ensures that
	\begin{equation}\label{eq:SLLN normal}
			S_M:=	\frac{1}{Q_M}\sum_{N=1}^M\sum_{k=1}^{K_N}b_{N,k}-\mu
                      \end{equation}
converges to $0$                      almost surely as $M\to\infty$, with $Q_M:= \sum_{N=1}^MK_N $.
Thus, from the identity
\begin{equation*}
	S_M =\frac{Q_{M-1}}{Q_M} S_{M-1}+\frac{K_M}{Q_M}\left(\frac{1}{K_M}\sum_{k=1}^{K_M}b_{M,k}-\mu\right)
      \end{equation*}
  and the fact that $Q_{M-1}/Q_M\leq 1$ we obtain
\begin{align*}
\limsup_{M\to\infty}\left|\frac{1}{K_M}\sum_{k=1}^{K_M}b_{M,k}-\mu\right|\leq \frac{\lim_{M\to\infty}(|S_M|+|S_{M-1}|)}{\liminf_{M\to\infty}\frac{K_M}{Q_M}}=0
\end{align*}
almost surely. Notice that we have used the assumption $\liminf_{M\to\infty}\frac{K_M}{Q_M}>0$. The lemma then follows.
\end{proof}

We are now ready to prove the main result of this section. Here and in
what follows, we shall use the notation
$ q\approx Q$ or $q\lesssim Q$ when there exists a constant $C$ (independent of the
large parameter under consideration) such that
$Q/C\leq q\leq C Q$ or $q\leq CQ$, respectively.

\begin{proposition}\label{P.reg}
  For each $\de>0$, the Gaussian random function~\eqref{deff} satisfies
\[
  f\in \Big[H^{s-\frac12-\de}(\TT)\backslash H^{s-\frac12}(\TT)\Big] \cap
  \Big[B^{s-\frac12}_{2,\infty}(\TT)\backslash B^{s-\frac12+\de}_{2,\infty}(\TT)\Big]
\]
  almost surely.
\end{proposition}

\begin{proof}
	Let us recall that the $H^\si(\TT)$ norm of the
        function~$f$ defined in~\eqref{deff} is
\begin{equation*}
	\|f\|_{H^\si(\TT)}^2=\sum_{ l =-\infty}^\infty |a_l|^2l^{2\si-2s}\,.
\end{equation*}
To analyze this quantity, consider the set of integers $\Lambda_N:=\{
l: 2^{N-1}\le l< 2^N \}$ and the subsequences
\begin{equation*}
	\sum_{ l =-(2^M-1)}^{2^M-1} |a_l|^2l^{2\si-2s}=|a_0|^2+2\sum_{N=1}^{M}\sum_{l\in\Lambda_N} l^{2\si-2s} |a_l|^2\approx |a_0|^2+\sum_{N=1}^{M}2^{N(2\si-2s)}\sum_{l\in\Lambda_N} |a_l|^2\,.
      \end{equation*}
      Since $|\La_N|\approx 2^N$,
      \[
        \frac{|\La_M|}{
          \sum_{N=1}^{M}|\La_N|} \approx \frac{2^M}{2^{M+1}}=\frac12
      \]
is bounded away from zero. Hence one can apply
      Lemma~\ref{lem:double SLLN} to infer that
$$
\frac{1}{|\La_N|}	\sum_{l\in\Lambda_N}  |a_l|^2\to1
$$
almost surely as $N\to\infty$. Therefore, with probability~1,
\begin{equation*}
	\sum_{ l =-(2^M-1)}^{2^M-1}
        |a_l|^2l^{2\si-2s} \approx |a_0|^2+\sum_{N=1}^{M}2^{N(2\si-2s+1)} \frac{1}{|\La_N|}	\sum_{l\in\Lambda_N} |a_l|^2\approx |a_0|^2+\sum_{N=1}^{M}2^{N(2\si-2s+1)}\,.
\end{equation*}
This shows that, with probability~1, $\|f\|_{H^\si(\TT)}<\infty$ if and
only if $\si<s-\frac12$.

The estimate for the Besov norm follows from an analogous reasoning
using that
\[
\|f\|_{B^\si_{2,\infty}(\TT)}^2= \sup_{1\leq N<\infty}
\sum_{l\in\La_N} l^{2\si-2s}|a_l|^2\,.
  \]
\end{proof}

\begin{remark}
The result and the proof remain valid in higher
dimensions with minor modifications. Specifically, let $\{Y_{lm}: 1\leq m\leq d_l,\; 0\leq
l<\infty\}$ be an orthonormal basis of spherical harmonics on the unit
$(n-1)$-dimensional sphere $\S^{n-1}$, with $\De_{\S^{n-1}}Y_{lm}+
l(l+n-2)Y_{lm}=0$. Consider the Gaussian random function
\[
f(x):= \sum_{l=1}^\infty\sum_{m=1}^{d_l} l^{-s} a_{lm} Y_{lm}(x)\,,
\]
where $a_{lm}$ are independent standard Gaussian variables and
$s\in\RR$. Then
\[
f\in \Big[H^{s-\frac{n-1}2-\de}(\S^{n-1})\backslash H^{s-\frac{n-1}2}(\S^{n-1})\Big] \cap
  \Big[B^{s-\frac{n-1}2}_{2,\infty}(\S^{n-1})\backslash B^{s-\frac{n-1}2+\de}_{2,\infty}(\S^{n-1})\Big]\]
almost surely.

To spell out the details, the proof in higher dimension starts with the formula
\[
\|f\|^2_{H^\si(\S^{n-1})}\coloneqq\sum_{ l =1}^\infty \sum_{m=1}^{d_l} |a_{ l  m}|^2l^{2\si-2s}\,.
\]
Since $d_l=c_n l^{n-2} + O(l^{n-3})$, the set
$$
\La_N:=\{(l,m): 2^{N-1}\le l< 2^N,\; 1\leq m\leq d_l\}
$$
satisfies $|\La_N|\approx 2^{N(n-1)}$. Lemma~\ref{lem:double SLLN} then ensures
$$
\frac{1}{|\La_N|}	\sum_{(l,m)\in\Lambda_N}  |a_{ l  m}|^2
$$
converges to~1 almost surely as $N\to\infty$, and the result follows
from the same argument as above. Obviously, the result also remains valid if one
replaces the weight $l^s$ by another quantity $w_l\approx l^s$.
\end{remark}

\section{Asymptotics for weighted Bessel series}
\label{S.Bessel}

In this section we shall prove Lemma~\ref{L.Bessel}. In view of the
well-known asymptotics
\[
J_l(r)=\left(\frac2{\pi r}\right)^{\frac12}\cos\bigg(r-\frac{(2l+1)\pi}4\bigg)+O(r^{-1})
\]
for Bessel functions, it is easy to check that the series
\begin{equation}\label{eq:Bessel series}
\cJ_{s,m,m'}(r)\coloneqq\sum_{l=1}^\infty l^{-2s} J_{l+m'}(r)\,J_{l+m}(r)\,.
\end{equation}
is locally uniformly convergent by the standard bound~\cite[(10.14.4)]{Olv10}
$$
|J_{ l }(r)|\le \frac{r^ l }{2^ l  l !}\,.
$$ 	
We are interested in the effect of the parameters $s\in\bR$ and $m',m\in\bZ$.

In view of the well-known integral representation
formula~\cite[(10.9.2)]{Olv10} for Bessel functions of integer order,
\[
J_l(r)=\frac1{2\pi}\int_{-\pi}^\pi e^{ir\sin x-ilx}\, dx\,,
\]
one can write
\begin{equation}\label{eq:def G}
	\cJ_{s,m,m'}(r)=\frac{1}{4\pi^2}\sum_{ l =1}^\infty l^{-2s} g_{\lambda_l}(r).
\end{equation}
Here we have set $\la_l:=l/r$,
\[
  g_\la(r):=\int_{-\pi}^\pi \int_{-\pi}^\pi e^{ir\varphi_{\lambda}(x,y)-i(m'x-my)}\,dx\,dy\,,
\]
and the phase function is
\[
  \varphi_{\lambda}(x,y)\coloneqq\lambda
(y-x)+\sin x-\sin y\,.
\]
Notice that we have used that $J_l$ is real valued, and hence $J_l=\overline{J_l}$.

A straightforward application of the stationary phase formula~\cite[Theorem 7.7.5]{Hor15} gives
the following asymptotic formula for~$g_\la$. Here and in what
follows, we will use the notation
\[
f(\lambda)\coloneqq\sqrt{1-\lambda ^2} -\lambda   \arccos
  \lambda\,,\qquad  \mu\coloneqq m+m'\,,\qquad \nu\coloneqq m-m'\,.
\]
Also, we will use the notation $O_p(r^{-k})$ to emphasize that a
certain quantity of order
$r^{-k}$ is not bounded uniformly with respect to the parameter~$p$.

\begin{lemma}\label{L.asymptgla}
  Suppose that $\la\neq1$. For
  $r\gg1$, one then has
  \[
    g_\lambda(r)= \frac{4 \pi  \left[\cos \left(\nu\arccos \lambda
    \right)+\sin \left(2rf(\lambda)  -\mu\arccos \lambda
    \right)\right]}{r {\left| 1-\lambda ^2\right|^{1/2}
  }}+O_{\lambda}(r^{-2})\,,
\]
where the error term is {not} bounded uniformy for large~$\la$
or for $\la$ close to~$1$.
\end{lemma}

\begin{proof}
For $\la\neq1$, the phase function $\vp_\la(x,y)$ has four critical
points
\[
 \{(x_i,y_i)\}_{i=1}^4:=\{ (\pm \arccos\lambda,\pm \arccos\lambda)\}
\]
with the same Hessian:
$$
\nabla^2\varphi_\lambda(x_i,y_i)=\left(
\begin{array}{cc}
	\mp\sqrt{1-\lambda ^2} & 0 \\
	0 & \pm\sqrt{1-\lambda ^2} \\
\end{array}
\right)\,.
$$
The stationary phase method~\cite[Theorem 7.7.5]{Hor15} then yields
\begin{align*}
	g_\lambda(r)&= \frac{2\pi}{r}\sum _{i=1}^4 e^{\frac{1}{4} i
                      \pi  \sigma_i} e^{i( m y_i- x_i m')} e^{i r
                      \varphi_\lambda (x_i,y_i)} \frac{1}{|\det
                      \nabla^2\varphi_\lambda (x_i,y_i)|}
                      +O_{\lambda}(r^{-2})\\[1mm]
  &= \frac{4 \pi  \left[\cos \left(\nu\arccos \lambda
    \right)+\sin \left(2rf(\lambda)  -\mu\arccos \lambda
    \right)\right]}{r {\left| 1-\lambda ^2\right|^{1/2}
  }}+O_{\lambda}(r^{-2})\,,
\end{align*}
as claimed. In this formula, $\sigma_i$ is the signature of the matrix $\nabla^2\varphi_\lambda(x_i,y_i)$.
\end{proof}

Therefore, the asymptotic analysis of~$g_\la(r)$ becomes problematic
when $\la$ is close to~1 (because in this case the phase function
presents degenerate or ``almost degenerate'' critical points) and when
$\la$ is large (because the error terms are not uniformly bounded in
this case). Consequently, we will
fix a small parameter $\de>0$ and consider smooth cutoff functions
$[0,\infty)\to[0,1]$ such that
\begin{align*}
  \chis (\la)&:= \begin{cases} 0 & \text{if } \la>1-\de\,,\\
1 & \text{if } \la<1-2\de\,,
  \end{cases}\\
\chil (\la)&:= \begin{cases} 0 & \text{if } \la<1+\de\,,\\
1 & \text{if } \la>1+2\de\,,
\end{cases}\\
  \chim(\la)&:= 1-\chis (\la)-\chil(\la)\,.
\end{align*}
We can then split $\cJ_{s,m,m'}(r)$ as
\[
\cJ_{s,m,m'}(r)=\frac{1}{4\pi^2}\left(\uno + \dos + \tres\right)
\]
with
\begin{align*}
  \uno:= \sum_{l=1}^\infty \chis (\la_l)\,l^{-2s} g_{\lambda_l}(r)\,,\qquad
    \dos:= \sum_{l=1}^\infty\chim(\la_l) \,l^{-2s} g_{\lambda_l}(r)\,,\qquad
  \tres:= \sum_{l=1}^\infty \chil(\la_l)\,l^{-2s} g_{\lambda_l}(r)\,.
\end{align*}
Note that $\uno$ only involves frequencies smaller than  $(1-\de) r$, $\dos$
involves frequencies close to~1 (more precisely, in the interval
$(1-2\de)r<l<(1+2\de)r$), and $\tres$ involves frequencies larger than $(1+\de)r$.

\subsection{The small frequency region}

In view of the asymptotic expansion for~$g_\la(r)$ proved in
Lemma~\ref{L.asymptgla},
it is natural to consider the closely related quantities
\begin{align*}
\uno'&:= \frac{4\pi}r\sum_{l=1}^\infty \chis (\la_l)\,
       \lambda_l^{-2s}\frac{\cos(\nu\arccos\la_l)}{(1-\la_l^2)^{1/2}}\,,\\
  \uno''&:= \frac{4\pi}r\sum_{l=1}^\infty \chis (\la_l)\,
       \lambda_l^{-2s}\frac{\sin(2r\,f(\la_l)-\mu\arccos\la_l)}{(1-\la_l^2)^{1/2}}\,.
\end{align*}
Lemma~\ref{L.asymptgla} obviously implies
\begin{equation}\label{unouno'uno''}
\uno= r^{-2s}(\uno'+\uno'')+ O_\de(r^{-2s-1})\,.
\end{equation}

Let us start by analyzing the large~$r$ behavior of~$\uno'$ when $s\leq\frac12$:

\begin{lemma}\label{L.uno'}
  For $r\gg1$ and some $\eta>0$ depending on~$s$,
  \[
    \uno'=\begin{cases}
\frac{4\pi^2  2^{2s -1} \Gamma (1-2s)}{\Gamma
  \left(1-s -\frac\nu2\right) \Gamma \left(
    1-s +\frac\nu2 \right)}+O(\delta^{\frac12})
+O_\de(r^{-\eta}) &\text{ if $s<\frac12$,}\\[1mm]
4\pi\cos \left(\frac{\pi  \nu }{2}\right) \log r +O_\de(1)&\text{ if
  $s=\frac12$ and $\nu$ is even,}\\[1mm]
2\pi^2\sin\left(\frac{\pi}{2}|\nu|\right)+O_\de(r^{-1})+O(\delta^{\frac12}) &\text{ if
  $s=\frac12$ and $\nu$ is odd.}
    \end{cases}
  \]
\end{lemma}

\begin{proof}
  Let us start with the case $s\leq0$. The basic observation here
  is that, as the function
  \[
    h(\lambda)\coloneqq4\pi\,\chis (\la)\,\frac{\lambda^{-2s}}{\sqrt{ 1-\lambda ^2}}
    {  \cos \left(\nu\arccos \lambda \right)}
  \]
  is H\"older continuous,
  \[
\frac1r\sum_{l=1}^{(1-\de )r}h(\la_l)=\int_{0}^{1-\de }h(\la)\,d\la + O_\de(r^{-1})
\]
by standard results about the convergence of Riemann sums for
integrands of bounded variation. If $s\leq0$, the result then follows from the formula
  \begin{equation}\label{integral}
	\int_0^1 \frac{\lambda^{-2s } \cos \left(\nu  \arccos\lambda\right)}{\sqrt{1-\lambda^2}} \, d\lambda=\frac{\pi  2^{2s -1} \Gamma ( 1-2s)}{\Gamma \left( 1-s-\frac\nu2\right) \Gamma \left(1-s +\frac\nu2\right)}
\end{equation}
and the estimate $\arcsin 1-\arcsin(1-\de)=O(\de^{1/2})$.

For $s\in(0,\frac12)$, the integrand is an unbounded function in
$L^1_{\mathrm{loc}}$, so the argument does not apply. Let us take a
small constant $\ep$ such that, for simplicity of notation, $\ep r$ is
an integer, and write
\[
\uno'= \frac1r\sum_{l=1}^{\ep r-1}h(\la_l)+\frac1r\sum_{l=\ep r}^{(1-\de)r}h(\la_l)=:\uno'_1+\uno'_2\,.
\]
Obviously, as $|h(\la)|\approx \la^{-2s}$ for small~$\la$, and $\int_{l/r}^{(l+1)/r}\la^{-2s}d\la\approx r^{-1}\la_l^{-2s}$, we conclude that
\[
\left|\uno'_1-\int_0^\ep h(\la)\,d\la\right|\lesssim \ep ^{2-2s}+r^{-1+2s}\,.
\]

To estimate $\uno'_2$, we use that
\begin{align*}
\uno'_2-\int_\ep^{1-\de} h(\la)\,d\la& = \sum_{l=\ep
                                                      r}^{(1-\de)r}\int_{(l-1)/r}^{l/r}[h(\la_l)-h(\la)]\,d\la=\frac1r\sum_{l=\ep r}^{(1-\de)r} {\frac{h'(\la_l^*)}{r}}
\end{align*}
for some $\la_l^*\in (\frac{l-1}r,\frac lr)$. Therefore, as
$|h'(\la)|\lesssim \la^{-2s-1}$,
\[
\left|\uno'_2-\int_\ep^{1-\de} h(\la)\,d\la\right|\lesssim \frac{\ep^{-1-2s}}r\,,
\]
where the constant in $\lesssim$ depends on $\de$.

Putting together the estimates for $\uno'_1$ and $\uno'_2$ with
$\ep\approx r^{-\frac12}$, we obtain
\[
\uno'= \int_0^{1-\de} h(\la)\,d\la + O_\de(r^{s-\frac12})= \int_0^{1} h(\la)\,d\la + O_\de(r^{s-\frac12})+O(\de^{1/2})\,.
\]
Using again the formula~\eqref{integral}, this proves the lemma when $s\in(0,\frac12)$.

Let us now pass to the case $s=\frac12$. We start by assuming that the
integer $\nu$ is odd, so that $\cos
        \left(\frac{\pi  \nu }{2}\right)=0$.
	Since
	\begin{equation}\label{identrig}
		\cos \left(\nu\arccos \lambda_l \right)=\cos \frac{\pi  \nu }{2}+\lambda_l  \nu  \sin \frac{\pi  \nu }{2}+O(\lambda_l ^2)\,,
	\end{equation}
	it turns out that the corresponding
        integrand is differentiable at $\la=0$ in this case, so the same arguments
        as in the case $s<0$ show
	\begin{equation*}
		\sum_{l=1}^{(1-\delta)r}\frac{\chis (\la_l)}{\lambda_l r}\frac{4 \pi  \cos \left(\nu\arccos \lambda_l \right)}{( 1-\lambda_l ^2)^{1/2}} =4\pi\int_0^{1-\delta} \frac{\chis (\la)\cos \left(\nu \arccos\lambda \right)}{\lambda  \sqrt{1-\lambda ^2}} \, d\lambda +O_\de(r^{-1})\,.
              \end{equation*}
The result then follows from the formula
\begin{equation*}
	\int_0^{1} \frac{\cos \left(\nu \arccos\lambda \right)}{\lambda  \sqrt{1-\lambda ^2}} \, d\lambda =\frac{\pi}{2}\sin\left(\frac{\pi}{2}|\nu|\right)\,.
      \end{equation*}

      To conclude, consider the case when $s=\frac12$ and $\nu$ is
      even. Obviously, by~\eqref{identrig},
      \begin{equation*}
	\frac{4\pi}{r}\left|\sum_{l=1}^{(1-\de)r}\chis (\la_l)\left(\frac{\cos \left(\nu\arccos \lambda_l
              \right)}{\lambda_l ({ 1-\lambda_l ^2})^{1/2}}-\frac{
              \cos \frac{\pi  \nu }{2}}{\lambda_l
            }\right)\right| \lesssim \frac1r\sum_{l=1}^{(1-\de)
          r}{\lambda_l}\lesssim 1\,,
      \end{equation*}
where the constant in $\lesssim$ depends on $\de$. The leading contribution of this sum is therefore given by the
      harmonic series, which satisfies
      \[
\sum_{l=1}^{(1-\de)
            r} \frac{\chis (\la_l)}{r\la_l} =\sum_{l=1}^{
            r/2} \frac{1}{l} +\sum_{l=\frac r2+1}^{(1-\de)
            r} \frac{\chis (\la_l)}{l} = \log r +O(1)\,.
        \]
        This completes the proof of the lemma.
\end{proof}

Now we pass to analyzing the contribution of the second term,
$\uno''$. As this term is somewhat oscillating due to the presence of
the large parameter~$r$ in the argument of a sine, it makes sense  to
expect this term should be subdominant.

\begin{lemma}\label{L.uno''}
  There exists some $\eta>0$, depending on~$s$, such that
  \[
    \uno''=\begin{cases}
      O_\de(r^{-\eta}) & \text{ if } s<\frac12\,,\\
      -4\pi\log 2\, \sin \left(2r  -\frac{\pi\mu}{2} \right)+O_\de(r^{-\eta}) & \text{ if } s=\frac12\,.
    \end{cases}
  \]
\end{lemma}

\begin{proof}
We start with the case $s<\frac12$. Let $\be\in(0,1)$ be some constant
that we will specify later and write
\[
\uno''=\Imag\left(\frac1r\sum_{l=1}^{\lfloor r^\be\rfloor} h(\la_l)\,
e^{i2rf(\lambda_l)} +\frac1r\sum_{l=\lceil r^\be\rceil}^{(1-\de)r} h(\la_l)\,
e^{i2rf(\lambda_l)}\right)=: \Imag(\uno''_1+\uno''_2)\,,
\]
with $h(\la):=4\pi\chis (\la)\la^{-2s}e^{-i\mu\arccos \lambda}( 1-\lambda^2)^{-\frac12}$. As $s<\frac12$, the first term can be easily estimated as
\[
|\uno''_1|\lesssim \frac1r\sum_{l=1}^{\lfloor r^\beta\rfloor}
\lambda_l^{-2s}\lesssim r^{-(1-2s)(1-\beta ) }\,.
\]
By hypothesis, the RHS is $r^{-\eta}$ for some $\eta>0$.

To estimate $\uno''_2$, decompose
the interval $(\lceil r^\be\rceil, (1-\de)r]$ as the union of
$N$ disjoint intervals of the form $(l_n,l_n+\La_n]$. We assume that
$l_n$ are integers and that
the lengths of the intervals satisfy $\La_n\approx r^\ga$ for some
$\ga\in(0,\be)$. This implies that $N\approx r^{1-\ga}$.

The basic idea is that, with this choice of the scales, one can expect that the
function~$h$ will be approximately constant in each interval but the
phase of the complex exponential will oscillate rapidly. This will
lead to cancellations. To make this idea precise, suppose that
$\la-\la_{l_n}\in(0,\La_n/r)$ and write
\begin{equation}\label{eq.flam}
f(\lambda)=:f(\lambda_{l_n})- (\la-\la_{l_n}) \arccos(\lambda_{l_n})+R_n(\la)\,,
\end{equation}
where the function $R_n(\la)$ plays the role of an error
term. Differentiating this identity with respect to~$\la$, and noticing that $f'(\la)=-\arccos \la$, one
immediately obtains that the bound $|R_n'(\la)|\lesssim |\la-\la_n|$
holds uniformly
in~$n$. As a consequence of this, setting $L:=r(\la-\la_{l_n})$, one
infers that
\begin{equation*}
	\left|\frac{d}{d\lambda}\left(h(\lambda)e^{i2rR_{n}(\lambda)}\right)\right|\le \left|h'(\lambda)\right|+\left|h(\lambda)2rR'_n(\lambda)\right|\lesssim r^{(\beta-1)(\al_1-1)}+r^{(\beta-1)\alpha_0}L
      \end{equation*}
      where
      \[
\al_0:=\min\{0,-2s\}\,,\qquad \al_1:= \min\{1,-2s\}\,.
      \]
As usual, the constant in $\lesssim$ depends on $\de$.

By the mean value theorem, observing that $R_n(\la_{l_n})=0$, one then has from Equation~\eqref{eq.flam} that
\begin{equation*}
	\left|\sum_{l=l_n+1}^{l_n+\La_n}\left(h(\lambda_l)e^{i2rf(\lambda_l)}-h(\lambda_{l_n})e^{i2rf(\lambda_{l_n})}e^{i2f'(\lambda_{l_n})L}\right)\right| \lesssim r^{(\beta-1)(\al_1-1)+2\gamma-1}+r^{(\beta-1)\alpha_0+3\gamma-1}\,,
\end{equation*}
with $\lesssim$ depending on $\de$. As the implicit constants are uniform in $n$ and there are $N\approx r^{1-\gamma}$
intervals, this implies
\begin{equation*}
	\uno''_2=\frac1r\sum_{n=1}^N
        h(\lambda_{l_n})e^{i2rf(\lambda_{l_n})}\sum_{L=0}^{\La_n}e^{i2f'(\lambda_{l_n})L}+O_\de(r^{(\beta-1)(\al_1-1)+\gamma-1}+r^{(\beta-1)\alpha_0+2\gamma-1})\,.
\end{equation*}
The leading contribution is therefore
\begin{align*}
	\frac1r\sum_{n=1}^N
        h(\lambda_{l_n})e^{i2rf(\lambda_{l_n})}\sum_{L=0}^{\La_n}e^{i2f'(\lambda_{l_n})L}&=
	\frac1r\sum_{n=1}^N h(\lambda_{l_n})e^{i2rf(\lambda_{l_n})}\frac{1-e^{-2                                                                                                 i \arccos(\lambda_{l_n} ) \left(r^{\gamma }+1\right)}}{1+e^{-2 i \arcsin(\lambda_{l_n} )}}\\
  &\lesssim r^{(\beta-1)\alpha_0-\gamma}\,,
\end{align*}
the constant in $\lesssim$ depending on $\de$. Note that the denominator is bounded from below because
$\lambda<1-\delta$. Thus, choosing $\gamma\in(0,\frac12)$ and $\beta$
sufficiently close to 1 (depending on $\gamma$ and $s$), we conclude that
\begin{equation*}
|\uno''_2|\lesssim r^{-\eta'}
\end{equation*}
for some $\eta'>0$.

Let us now pass to the case $s=\frac12$. Arguing  as above, one can
pick some $\beta$ close to, but smaller than, 1 such that
\begin{equation*}
			\sum_{l=\lceil r^\beta\rceil}^{(1-\delta^-)r}\frac{\chis (\la_l)\sin \left(2rf(\lambda_l)  -\mu\arccos \lambda_l \right)}{l ( 1-\lambda_l ^2)^{1/2}}=O_\de(r^{-\eta})
	\end{equation*}
        for some $\eta>0$.
For the sum going from $l=1$ to $\lfloor r^\beta\rfloor$, we can
disregard the $(1-\lambda_l^2)^{1/2}$ term because
\begin{multline*}
	\left|\sum_{l=1}^{\lfloor r^\beta \rfloor}\left[\frac{\sin
              \left(2rf(\lambda_l)  -\mu\arccos \lambda_l
              \right)}{l (1-\lambda_l ^2 )^{1/2}}-\frac{\sin
              \left(2rf(\lambda_l)  -\mu\arccos \lambda_l
              \right)}{l }\right]\right|
        \lesssim \sum_{l=1}^{\lfloor r^\beta \rfloor}
        \frac{\la_l}r\lesssim r^{-2+2\be}\,.
      \end{multline*}
The identity
      \begin{multline*}
	\sin \left(2rf(\lambda_l)  -\mu\arccos \lambda_l \right)=\sin \left(2r  -\frac{\pi\mu}{2} \right)\cos\left(2r(f(\lambda_l)-1)+\mu\left(\frac{\pi}{2}-\arccos\lambda_l\right)\right) \\
	+\cos \left(2r  -\frac{\pi\mu}{2} \right)\sin\left(2r(f(\lambda_l)-1)+\mu\left(\frac{\pi}{2}-\arccos\lambda_l\right)\right)
      \end{multline*}
      enables us to write
      \begin{multline*}
\sum_{l=1}^{\lfloor r^\beta \rfloor}\frac{\sin
              \left(2rf(\lambda_l)  -\mu\arccos \lambda_l
              \right)}{l }= \sin \left(2r  -\frac{\pi\mu}{2} \right)
            \sum_{l=1}^{\lfloor r^\beta
              \rfloor}\frac{\cos\left(2r(f(\lambda_l)-1)+\mu\left(\frac{\pi}{2}-\arccos\lambda_l\right)\right)}{l
            }\\
            +\cos \left(2r  -\frac{\pi\mu}{2} \right)\sum_{l=1}^{\lfloor r^\beta
              \rfloor}\frac{\sin\left(2r(f(\lambda_l)-1)+\mu\left(\frac{\pi}{2}-\arccos\lambda_l\right)\right)}l\,.
          \end{multline*}
The  asymptotic expansions
\begin{equation*}
	f(\lambda)-1=-\frac{\pi  \lambda }{2}+O(\lambda ^2),\quad \frac{\pi }{2}-\arccos\lambda =\lambda   +O(\lambda ^2)
\end{equation*}
ensure that
$$
2r(f(\lambda_l)-1)+\mu\left(\frac{\pi}{2}-\arccos\lambda_l\right)=-\pi l+r O(\lambda^2)\,.
$$
The quantity $r O(\lambda^2)$ is of order $r^{2\beta'-1}$ whenever
$l<r^{\beta'}$. Fixing some $\be'\in(0,\frac12)$, we therefore have
\begin{align*}
	\sum_{l=1}^{\lfloor r^{\beta'}
              \rfloor}\frac{\cos\left(2r(f(\lambda_l)-1)+\mu\left(\frac{\pi}{2}-\arccos\lambda_l\right)\right)}{l
            }&=\sum_{ l =1}^{\lfloor r^{\beta'}
              \rfloor}\left(\frac{\cos(\pi l)}{l}+
               \frac{r^2O(\la_l^4)}l\right)\\
  &=-\log 2+O(r^{-\min\{{\beta'},2-4\be'\}})\,.
            \end{align*}
Here we have used that
\begin{equation*}
	\sum _{l=1}^L \frac{\cos(\pi l)}{l}=-\log 2+ O(L^{-1})\,.
\end{equation*}
Similarly,
\begin{equation*}
	\sum_{l=1}^{\lfloor r^{\beta'}
              \rfloor}\frac{\sin\left(2r(f(\lambda_l)-1)+\mu\left(\frac{\pi}{2}-\arccos\lambda_l\right)\right)}{l
            }= \sum_{l=1}^{\lfloor r^{\beta'}
              \rfloor}\frac{rO(\la_l^2)}l=O(r^{2\be'-1})\,.
          \end{equation*}

   It only remains to consider the sum from $\lceil r^{\beta'} \rceil$
   to $\lfloor r^\beta\rfloor$, where we can also assume that $\chis (\la_l)=1$. To this end, we define the function
\begin{equation*}
Q:=\sum_{l=\lceil r^{\beta'} \rceil}^{\lfloor r^\beta \rfloor}\frac{e^{i\left(2r(f(\lambda_l)-1)+\mu\left(\frac{\pi}{2}-\arccos\lambda_l\right)\right)}}{l}{=:\sum_{l=\lceil r^{\beta'} \rceil}^{\lfloor r^\beta \rfloor}\frac{e^{-i\left(\pi l+\varphi(\lambda_l,r)\right)}}{l}}\,.
\end{equation*}
To show this sum goes to zero as $r\to\infty$, we are going to exploit
the cancellations of consecutive terms. For this, let us define
\begin{align*}
\De_{2k}&:=\varphi(\lambda_{2k+1},r)-\varphi(\lambda_{2k},r)\\
&\phantom{:}=2r(f(\lambda_{2k})-1)+\mu\left(\frac{\pi}{2}-\arccos\lambda_{2k}\right)-\left[2r(f(\lambda_{2k+1})-1)+\mu\left(\frac{\pi}{2}-\arccos\lambda_{2k+1}\right)\right]-\pi\,.
\end{align*}
More explicitly,
\begin{equation*}
	\Delta_{2k}=2 \sqrt{r^2-4k^2}-2 \sqrt{r^2-(2k+1)^2}-(4k+\mu ) \arccos \left(\frac{2k}{r}\right)+(4k+\mu +2) \arccos \left(\frac{2k+1}{r}\right)-\pi.
\end{equation*}
By the mean value theorem, there exists some $\la_*\in(2kr^{-1},(2k+1)r^{-1})$ such that
\begin{equation*}
|\Delta_{2k}|\le\left|\pi r-2\arccos \la_*+\frac{\mu}{r}(1-\la_*^2)^{-1/2}\right|r^{-1}\lesssim \frac lr
\end{equation*}
for $\lceil r^{\beta'} \rceil<l<\lfloor r^\beta \rfloor$. This enables
us to estimate~$Q$ as
\begin{align*}
|Q|&=\left|\sum_{k=\lceil r^{\beta'} \rceil/2}^{\lfloor r^\beta
     \rfloor/2} e^{i
     2r(f(\lambda_{2k})-1)+\mu\left(\frac{\pi}{2}-\arccos\lambda_{2k}\right)}
     \left(\frac1{2k}- \frac{e^{-i\De_{2k}}}{2k+1}\right)\right|\\
  &\lesssim \sum_{k=\lceil r^{\beta'} \rceil/2}^{\lfloor r^\beta     \rfloor/2}\left(\frac{1}{k^2}+\frac{1}{r}\right)\lesssim r^{-\beta'}+ r^{\beta-1}\,.
\end{align*}
\end{proof}

Let us finally consider the case $s>\frac12$:

\begin{lemma}\label{lem:G1 alpha le -1}
  If $s>\frac12$, there exists some $\eta>0$ depending on~$s$
        such that
	\begin{equation*}
		\uno=\frac{1}{\pi r}  \zeta (2s ) \left(\cos \frac{\pi  \nu}{2} -\left(2^{1-2s}-1\right) \sin \frac{\pi \mu-4 r}{2}\right)+O_\de(r^{-1-\eta})\,.
	\end{equation*}
        Here $\zeta$ is the Riemann's zeta function.
      \end{lemma}
      \begin{proof}
        Let us use again the integral formula for Bessel functions to write
        \[
        J_{l+m'}(r)J_{l+m}(r)=\frac1{4\pi^2}\int_{-\pi}^\pi \int_{-\pi}^\pi e^{ir(\sin
          x-\sin y)}\, e^{-i((l+m')x-(l+m)y)}\, dx\, dy\,.
\]
        Applying the stationary phase argument~\cite[Theorem
        7.7.5]{Hor15} with phase function $\sin x-\sin y$ and
        amplitude $e^{-i((l+m')x-(l+m)y)}$, one readily obtains the
        asymptotic expansion
	\begin{equation*}
	J_{l+m'}(r)J_{l+m}(r)=\frac{\cos \left(\frac{1}{2} \pi  \nu\right)-\sin \left(\frac{1}{2} \left(2 \pi  l+\pi \mu -4 r\right)\right)}{\pi r}+R_l(r)\,,
	\end{equation*}
        where the error term satisfies the pointwise bound
        \begin{equation*}
	\left|R_l(r)\right|\lesssim \frac{l^{4}}{r^2}\,.
\end{equation*}

Now, pick some $\beta\in(0,\frac{1}{4})$ and write
\begin{align*}
	\uno=&\sum_{l=1}^{\lfloor r^\beta \rfloor}l^{-2s} J_{l+m'}(r)J_{l+m}(r)+
	\sum_{l=\lceil r^\beta \rceil}^{(1-\de )r}\chis (\la_l)l^{-2s}
               J_{l+m'}(r)J_{l+m}(r)=:\uno_1+\uno_2\,.
\end{align*}
Then
\begin{align*}
\uno_1&=\sum_{l=1}^{\lfloor r^\beta \rfloor}l^{-2s}\left[\frac{\cos
        \left(\frac{1}{2} \pi  \nu\right)-\sin \left(\frac{1}{2}
        \left(2 \pi  l+\pi \mu -4 r\right)\right)}{\pi
        r}+R_l(r)\right]\\
  &=:\sum_{l=1}^{\lfloor r^\beta \rfloor}l^{-2s}\frac{\cos \left(\frac{1}{2} \pi  \nu\right)-\sin \left(\frac{1}{2} \left(2 \pi  l+\pi \mu -4 r\right)\right)}{\pi r}+\cR\,,
\end{align*}
where the error term is bounded as
\[
|\cR|=\left|\sum_{l=1}^{\lfloor r^\beta \rfloor}l^{-2s} R_l(r)\right|\lesssim
\frac1{r^2}\sum_{l=1}^{\lfloor r^\beta \rfloor}l^{4-2s}\lesssim r^{-2}(1+r^{\be(5-2s)})\,.
\]
This decay is smaller than~$r^{-1}$ if $\be<\frac{1}{4}$. Expanding the sine, the above series can be computed in closed form in terms of the zeta
function:
\[
\uno_1=\frac{1}{\pi r}  \zeta (2s ) \left[\cos \left(\frac{1}{2} \pi  \nu\right)-\left(2^{1-2s}-1\right) \sin \left(\frac{1}{2} \left(\pi \mu-4 r\right)\right)\right] +O(r^{-2}+r^{\be(5-2s)-2})\,.
\]

To control the remaining term, we use that $s>\frac12$ and the
bound for~$g_\la$ proved in Lemma~\ref{L.asymptgla} to write
\begin{align*}
	|\uno_2|&\lesssim \left| \sum_{l=\lceil r^\beta
                  \rceil}^{(1-\de )r}
                  \chis (\la_l)\,l^{-2s}g_{\la_l}(r)\right|\lesssim \frac{1}{r}\sum_{l=\lceil r^\beta\rceil
    }^{(1-\de )r}l^{-2s}\le \frac{1}{r}\sum_{l=\lceil
    r^\beta\rceil }^\infty l^{-2s} \lesssim r^{-\beta(2s-1)-1}\,.
\end{align*}
As usual, the constant in $\lesssim$ depends on $\de$. The lemma then follows.
\end{proof}

\subsection{Intermediate frequency region}

Our next goal is to derive bounds for the term
\[
\dos=\sum_{ l ={\lceil(1-2\de)r\rceil}}^{\lfloor(1+2\de)r\rfloor}\chim (\la_l)\,
l^{-2s} \int_{-\pi}^\pi \int_{-\pi}^\pi  e^{i l(y-x)} e^{ir(\sin x-\sin y)}e^{i(my-m'x)}\,dx\,dy\,.
\]
The difficulty here is that one cannot apply the standard stationary
phase method as we did above because the critical
points of the phase function
\[
  \varphi_l(x,y):=\lambda (y-x)+\sin
  x-\sin y
\]
are either degenerate or not uniformly non-degenerate. The
main result is the following:

\begin{lemma}\label{L.dos}
  For any real~$s$ and all large enough~$r$ (depending on~$\de$),
  \[
|\dos|\le C\de^{\frac12}\, r^{-2s}\,,
  \]
  where $C$ is independent of~$\de$.
\end{lemma}

\begin{proof}
  Since
  \[
\varphi_l(x,y)= (1-\la)(x-y) - \frac16(x^3-y^3)+O(x^5)+O(y^5)\,,
    \]
    when $1-2\de\leq \la\leq 1+2\de$ and $\de\ll1$, an elementary calculation
    shows that
    \[
|\nabla \varphi_l(x,y)|\geq c
    \]
   whenever $|x|+|y|>100\,\de^{1/2}$, where $c$ is a positive constant that
   depends on~$\de$. Therefore, take some $\chi(t)$ be a smooth nonnegative
   function that is equal to~1 for $|t|<100\,\de^{1/2}$ and~0 for
   $|t|>200\,\de^{1/2}$. The non-stationary phase lemma then shows
   that
   \[
\dos':=\sum_{ l ={\lceil(1-2\de)r\rceil}}^{\lfloor(1+2\de)r\rfloor}\chim (\la_l)\,
\lambda_l^{-2s} \int_{\RR^2}  e^{i l(y-x)} e^{ir(\sin x-\sin
  y)}e^{i(my-m'x)}\,\chi(x)\, \chi(y)\,dx\, dy
   \]
   coincides with~$\dos$ modulo an exponentially small error. More precisely,
   \[
\left|\dos-r^{-2s}\dos'\right|< C_{\de,N}\, r^{-N}
   \]
   for any~$N$ and some constant depending on~$N$ and~$\de$.

   To estimate $\dos'$, let us start by defining $z:= y-x$ and writing
   \[
\dos'=\sum_{l=r(1-2\de)}^{r(1+2\de)} \chim (\la_l)\,\la_l^{-2s} \int_{\bR^2} e^{i lz} e^{ir(\sin (y-z)-\sin y)}e^{i((m-m')y+m'z)}\chi(y-z)\,\chi(y)\,dy\,dz\,.
\]
A first step is to consider the sum
\[
S(r,z):=   \frac1r\sum_{l=r(1-2\de)}^{r(1+2\de)} \chim (\la_l)\,\la_l^{-2s} e^{i l z}\]
and to relate it to its continuous counterpart
\[
F(r,z)\coloneqq \int_{-\infty}^\infty\chim (\la)\,\lambda^{-2s} e^{irz\lambda}d\lambda\,.
\]
Note that it is not a priori obvious that $F(r,z)$ converges to $S(r,z)$ as
$r\to\infty$ because, intuitively speaking, the sum is formally obtained
by discretizing the integral with a ``grid'' of length~$1/r$, and $r\gg1$
is precisely the frequency at which the integrand oscillates.

We proceed as follows. Firstly, write
\begin{multline*}
	S(r,z)-F(r,z)=\sum_{l=r(1-2\de)}^{r(1+2\de)}\int_{\lambda_l}^{\lambda_l+\frac1r}
        \Big[\lambda_l^{-2s}\chim(\la_l) \Big(\frac{e^{i r \lambda_l z}}{r}- e^{i r
            \lambda z}\Big)\\
          +(\chim(\la_l)\lambda_l^{-2s}-\chim(\la)\lambda^{-2s}) e^{i r \lambda z}\Big]\,d\lambda
\end{multline*}
and note that
\begin{equation*}
	\frac{e^{i l z}}{r}-\int_{\lambda_l}^{\lambda_l+\frac{1}{r}} e^{i \lambda  r z} \, d\lambda=h(z)\,\frac{ e^{i l z}}{r }
\end{equation*}
with
\[
h(z):=\frac{i e^{i z}+z-i}{z}\,.
\]
The function~$h$ is smooth at the origin; in fact, $h(z)=O(z)$. As
moreover
\begin{equation}\label{eq:bound second dif}
|\chim(\la_l)\lambda_l^{-2s}-\chim(\la)\lambda^{-2s}|\lesssim\frac{\de^{-1}}{r}
\end{equation}
if $\la\in[\la_l,\la_l+\frac1r]$ and $|\la-1|<2\de$,
one obtains that the error
\[
R(r,z):= S(r,z)-F(r,z)-h(z)S(r,z)
\]
is bounded as
\[
|R(r,z)|\leq \frac{C}{r}\,,
\]
with~$C$ a constant independent of~$z$ and~$\de$.

Since~$z$ will eventually be small, the fact that
\begin{equation*}
	S(r,z)=\frac{F(r,z)+R(r,z)}{1-h(z)}
\end{equation*}
shows in which
sense~$S(r,z)$ and~$F(r,z)$ are related. The reader can check that, had we
argued as in~\eqref{eq:bound second dif}, we would have obtained an
error estimate of the form $Cz$, which is useless for our purposes.

One can thus write
\begin{multline*}
  \dos'=r\int_{\bR^3}\chim(\la)\,\lambda^{-2s}
e^{ir(\lambda z+\sin (y-z)-\sin
  y)}e^{i((m-m')y+m'z)}\frac{\chi(y-z)\chi(y)}{1-h(z)}\,d\lambda
\,dz\, dy\\
	+r\int_{\bR^2} e^{ir(\sin (y-z)-\sin
          y)}e^{i((m-m')y+m'z)}R(r,z)\frac{\chi(y-z)\chi(y)}{1-h(z)}\,dz\,dy\\
        =:\dos'_1+\dos'_2\,.
      \end{multline*}
      The bound for~$R(r,z)$ and the fact that $\chi(t)$ is supported in $|t|<200\,\de^{1/2}$ immediately implies
      \[
|\dos'_2|\leq C\de\,,
\]
where the constant does not depend on~$\de$.

      To analyze $\dos'_1$, one cannot directly apply the stationary
      phase formula to the integral over~$\RR^3$ because the critical
      set of the phase has dimension~1. Instead, let us define
      \[
H(r,y):= r\int_{\RR^2}e^{ir(\lambda z+\sin (y-z))}\chim(\la)\,\lambda^{-2s} e^{im'z} \frac{\chi(y-z)}{1-h(z)}d\lambda\, dz\,.
      \]
Then, the phase function $\varphi_y(\la,z):=\lambda z+\sin (y-z)$ has a
unique critical point in the support of the integrand, $(\la^*,z^*):=(\cos y,0)$, and
its Hessian is
$$
\nabla^2\vp_y(\la^*,z^*)=\left(
\begin{array}{cc}
	0 & 1 \\
	1 & -\sin(y) \\
\end{array}
\right)\,.
$$
      The stationary phase formula~\cite[Theorem 7.7.6]{Hor15} then
      ensures that, if $r$ is large enough (depending on $\de$)
      \[
|H(r,y)|\le C
      \]
      with a constant independent of~$\de$. Plugging this estimate
      into~$\dos_1'$ and using again that $\chi(t)$ is supported in $|t|<200\,\de^{1/2}$, one finds
      \[
|\dos'_1|\leq\int_{-\infty}^\infty \chi(y)\,|H(r,y)|\,dy\leq C\de^{\frac12}
      \]
      with a constant independent of~$\de$. Putting all the estimates
      together, the lemma is proven.
\end{proof}

\subsection{Large frequency region}

The last lemma of this section shows that the contribution of the
large frequencies is exponentially small:

\begin{lemma}\label{L.tres}
  For any~$N$,
  $
|\tres|\lesssim r^{-N}
$
for all large enough $r$ (depending on $\de$).
\end{lemma}

\begin{proof}
Let us now use~$l$ as the large parameter in the formula
for~$g_{\la_l}(r)$, which amounts to writing
\[
g_{\la_l}(r)= \int_{-\pi}^\pi\int_{-\pi}^\pi e^{il\tilde\varphi_{\la_l}(x,y)}e^{-i(m'x-my)}\,dx\,dy
\]
with
\[
\tilde{\varphi}_{\lambda}(x,y):=y-x+\frac{\sin x-\sin y}\la\,.
\]
If $\la>1+\de$, it is clear that
\[
|\nabla \tilde{\varphi}_{\lambda}(x,y)|\geq c_\de
\]
for all $x,y\in[-\pi,\pi]$, where $c_\de$ is a positive constant
that only depends on~$\de$. Therefore, the non-stationary phase
lemma~\cite[Theorem 7.7.1]{Hor15} ensures that $g_{\la_l}(r)$ is an
exponentially small function of~$l$, meaning that for any~$N'$ there exists a
constant~$C$ (depending on~$\de$ and~$N'$) such that
\[
  |g_{\la_l}(r)|<C|l|^{-    N'}\,.
\]
This immediately implies that
\[
|\tres|\lesssim\sum_{l=(1+\de)r}^\infty l^{-2s}|g_{\la_l}(r)|\lesssim r^{-N}
\]
for any~$N$, as claimed.
\end{proof}

\subsection{Asymptotics for series with derivatives of Bessel functions}

The results we have derived above readily yield the asymptotic
bounds for weighted sums of Bessel functions that we will crucially
need in the next section. Specifically, Lemma~\ref{L.Bessel} follows immediately by adding the estimates derived in the previous subsections and
letting $\de\to0^+$. The explicit constants in the lemma are:
\begin{equation*}
	\begin{aligned}
&c^1_{s,\nu}:=\frac{2^{2s-1}\Gamma(1-2s)}{\Gamma(1-s-\frac{\nu}{2})\Gamma(1-s+\frac{\nu}{2})}\,,\\
&c^2_\nu:=\pi^{-1}\cos\Big(\frac{\pi \nu}{2}\Big)\,,\\
&c^3_\nu:={2}^{-1}\sin\Big(\frac{\pi |\nu|}{2}\Big)\,,
\end{aligned}
\qquad\qquad
\begin{aligned}
&c^4:=\frac{\log 2}{\pi}\,,\\
&c^5_{s,\nu}:=\pi^{-1}\zeta(2s)\cos\Big(\frac{\pi \nu}{2}\Big)\,,\\
&c^6_{s,\nu}:=\pi^{-1}\zeta(2s)(1-2^{1-2s})\,,\\
&c^7_{\mu}:=\frac{\pi\mu}{2}\,.
\end{aligned}
\end{equation*}
One should observe that, to estimate the expected number of critical
points of the random monochromatic wave~\eqref{defu}, we will also
need asymptotic information about series with derivatives of Bessel
functions. This follows easily as a byproduct of Lemma~\ref{L.Bessel} using
the well-known recurrence relations
	\begin{align*}
		J'_l(r)=\frac{J_{l-1}(r)-J_{l+1}(r)}{2} \,,\qquad 
		J''_l(r)=\frac{J_{l+2}(r)+J_{l-2}(r)-2 J_l(r)}{4}
                          \,. 
	\end{align*}
        In the following
        lengthly corollary of Lemma~\ref{L.Bessel} we record the asymptotic
        formulas that we will need later on:
        \begin{corollary}\label{C.formulas}
          The following estimates hold:
          \begin{align*}
\sum _{l=1}^{\infty } l^{{-2s} } J_l(r)^2&=
            \begin{cases}
    \frac{2^{2s-1} \Gamma (1-2s)
      r^{-2s }}{\Gamma \left(1-s\right)^2}+o(r^{-2s}) &\text{if } s<\frac12,\\
    \frac{\log r }{\pi  r}+o(r^{-1})    &\text{if } s=\frac12,\\
 \frac{\zeta (2s ) \left(\left(2^{1-2s}-1\right) \sin2r+1\right)}{\pi  r}+o(r^{-1})   &\text{if } s>\frac12,
  \end{cases}\\
        \sum _{l=1}^{\infty } l^{{-2s} } J_l(r)
  J'_l(r)&= \begin{cases}
    o(r^{-2s}) &\text{if } s<\frac12,\\
    O(r^{-1})    &\text{if } s=\frac12,\\
 \frac{\left(2^{1-2s}-1\right) \cos (2 r) \zeta (2s )}{\pi  r}+o(r^{-1})  &\text{if } s>\frac12,
  \end{cases}\\
  \sum _{l=1}^{\infty } l^{{-2s} } J'_l(r)^2&= \begin{cases}
    \frac{\Gamma \left(\frac{1}{2}-s\right) r^{{-2s} }}{4
      \sqrt{\pi } \Gamma \left(2- {s}\right)}+o(r^{-2s}) &\text{if } s<\frac12,\\
    \frac{\log r}{\pi  r}+O(r^{-1})   &\text{if } s=\frac12,\\
 \frac{\zeta (2s) \left(1-\left(2^{1-2s}-1\right) \sin2r\right)}{\pi  r}+o(r^{-1})  &\text{if } s>\frac12,
\end{cases}\\
            \sum _{l=1}^{\infty } l^{{-2s} } J_l(r) J''_l(r)&=
                             \begin{cases}
                                                       -\frac{\Gamma
                                                         \left(\frac
                                                           12-s\right)
                                                         r^{-2s
                                                         }}{4
                                                         \sqrt{\pi }
                                                         \Gamma
                                                         \left(2-s\right)}+o(r^{-2s})
                                                       &\text{if }
                                                       s<\frac12,\\
                                                       -\frac{\log r}{\pi  r}+O(r^{-1})   &\text{if } s=\frac12,\\
 -\frac{\zeta (2s) \left(\left(2^{1-2s}-1\right) \sin2r+1\right)}{\pi  r}+o(r^{-1})  &\text{if } s>\frac12,
                              \end{cases}
  \\
  \sum _{l=1}^{\infty } l^{{-2s} } J'_l(r) J''_l(r)&=
                                                      \begin{cases}
                                                        o(r^{-2s})
                                                        &\text{if }
                                                        s<\frac12,\\
                                                        O(r^{-1})    &\text{if } s=\frac12,\\
 -\frac{\left(2^{1-2s}-1\right) \cos (2 r) \zeta (2s )}{\pi  r}+o(r^{-1})  &\text{if } s>\frac12,
                                                      \end{cases}
  \\
  \sum _{l=1}^{\infty } l^{-2s } J''_l(r)^2&=
                                                \begin{cases}
                                                  \frac{3\ 2^{2s
                                                      -5} (2-2s)
                                                    (4-2s) \Gamma
                                                    (1-2s)
                                                    r^{-2s
                                                    }}{\Gamma
                                                    \left(3-s\right)^2}+o(r^{-2s})
                                                  &\text{if }
                                                  s<\frac12,\\
                                                   \frac{ \log r}{\pi  r}+O(r^{-1})    &\text{if } s=\frac12,\\
 \frac{\zeta (2s ) \left(\left(2^{1-2s}-1\right) \sin (2
     r)+1\right)}{\pi  r}+o(r^{-1})  &\text{if } s>\frac12.
  \end{cases}
            \end{align*}
          \end{corollary}
\section{Proof of  Theorem~\ref{T.main}}
\label{S.main}

We are now ready to present the proof of the main theorem, which will
consist of a number of steps. Recall that we defined the random function~$u$ as
\begin{equation}\label{defu2}
u := \sum_{l} a_l \,\si_l\, e^{i l\te}\, J_l(r)\,,\qquad \si_l:= \begin{cases}|l|^{-s} & \text{if }l\neq0\,,\\ 0& \text{if }l=0\,.\end{cases}
\end{equation}
It will be apparent from the proof that
 the argument remains valid for much more general choices of $\si_l$, for example of the form~\eqref{defsil}.
Of course, the value of the constants $\ka(s)$, $\tka_{\frac32}$, $\tka_{\frac52}$ one gets  depends on the specific choice of~$\si_l$.

\subsection{A Kac--Rice formula}

Our first objective is to derive an explicit, if hard to analyze,
Kac--Rice type formula
for the expected number of critical points of the Gaussian random
function~$u$.

In this subsection, we shall denote by
$$
	D u(r,\theta)\coloneqq\left(
	\begin{array}{cc}
		\partial_\theta u(r,\theta) \\
		\partial_r u(r,\theta) \\
	\end{array}
      \right)\,,\qquad D^2 u(r,\theta)\coloneqq\left(
	\begin{array}{cc}
		\partial_{\theta\theta} u(r,\theta)&\partial_{r\theta} u(r,\theta)\vspace{1mm} \\
		\partial_{r\theta} u(r,\theta) &\partial_{rr} u(r,\theta) \\
	\end{array}
	\right)
        $$
        the derivative and Hessian of~$u$ in polar coordinates. To
        apply the Kac--Rice expectation formula, let us start by
        showing that $Du(r,\te)$ has a non-degenerate distribution:

        \begin{lemma}\label{LemmaNonDeg}
The variance of the Gaussian random variable $Du(r,\te)$ is
\[
\var[Du(r,\te)]=\left(
		\begin{array}{cc}
			4 \sum _{l=1}^{\infty }  l^{2-2s}   J_l(r)^2 & 0 \\
			0 & 4  \sum _{l=1}^{\infty } l^{-2s}   J'_l(r)^2 \\
		\end{array}\right)=:\left(\begin{array}{cc}
			\widetilde\Sigma_{11}(r) & 0 \\
			0 & \widetilde\Sigma_{22}(r)\\
		\end{array}\right)\,.
\]
\end{lemma}
\begin{proof}
To compute the matrix
\[
\var[Du(r,\te)]:=\bE[Du(r,\te)\otimes Du(r,\te)]\,,
\]
recall the expression~\eqref{defu2} for~$u(r,\te)$ and take advantage of the fact that $u(r,\te)$ is real valued to
write
\[
\bE[\pd_ru(r,\te)^2]=\bE[\pd_ru(r,\te)\,
\pd_r\overline{u(r,\te)}]=\sum_{l\neq0}\sum_{l'\neq0}
\bE(a_l\overline{a_{l'}})\, |l|^{-s}| l'|^{-s} e^{i(l-l')\te}J'_l(r)\,J'_{l'}(r)\,.
\]
By the definition of the random variables~$a_l$,
\[
\bE(a_l\overline{a_{l'}})= 2\de_{l,l'}\,,
\]
so one obtains
\[
\bE[\pd_ru(r,\te)^2]=4 \sum_{l=1}^\infty l^{-2s} J'_l(r)^2
\]
The same argument yields
\begin{align*}
\bE[\pd_ru(r,\te)\,\pd_\te u(r,\te)]&=\bE[\pd_\te u(r,\te)\,
                      \pd_r\overline{u(r,\te)}]\\
  &=\sum_{l\neq0}\sum_{l'\neq0}
\bE(a_l\overline{a_{l'}})\, il|l|^{-s}| l'|^{-s}
    e^{i(l-l')\te}J_l(r)\,J_{l'}'(r)\\
  &=2i\sum_{l\neq0} l|l|^{-2s} J_l(r)\, J_l'(r)=0
\end{align*}
by parity, and
\begin{align*}
  \bE[\pd_\te u(r,\te)^2]= 4  \sum _{l=1}^{\infty } l^{2-2s}   J_l(r)^2\,.
\end{align*}
This easily implies that $\var[Du(r,\te)]$ is a strictly positive matrix for
all~$(r,\te)$.
\end{proof}

\begin{remark}\label{KernelNS}
The same computation as above shows that the covariance kernel of the random function~\eqref{defu2} is
\[
K(r,\te;r',\te'):=\bE[u(r,\te)\,u(r',\te')]=4\sum _{l=1}^\infty l^{-2s} J_l(r) J_l(r') \cos [l(\theta-\theta')]\,.
\]
The covariance kernel is therefore invariant under rotations but, in general, not under translation. An exception to this general fact is the case $s=0$. Indeed, it is well known that the covariance kernel of
\begin{equation*}
 	\widetilde{u}:= u+\sqrt{2}a_0\, J_0(r).
\end{equation*}
is {$\widetilde{K}(x;x')=2J_0(|x-x'|)$} by Graf's Addition Theorem. The corresponding spectral measure in
this case is the Hausdorff measure on the unit circle. Observe that $\widetilde{u}$ will give the same asymptotics as $u$ for $s=0$ because, as we saw in Lemma \ref{L.Bessel}, for $s=0$ the series of Bessel functions is asymptotically of order~1 but the term $J_0(r)^2$ decays like $r^{-1}$. By Lemma \ref{L:KR est}, their covariances $\Sigma_{ij}$ are then asymptotically equivalent. Note we have chosen to omit the term $l=0$ in $u$ for simplicity, especially when this term contributes to the asymptotic expansion (that is, for $s>\frac12$ in Lemma~\ref{L.Bessel}).
\end{remark}

\begin{lemma}\label{L:KR est}
The expected value of the number of critical points of the random
monochromatic wave~\eqref{defu} is
	\begin{equation*}
          \bE N(\nabla u, R)=
          \int_0^{R} \int_{\RR^3}\frac{\left| z_1^2 \Sigma_{13}(r)-z_2^2 \Sigma_{22}(r)+z_3 z_1  \sqrt{\Sigma_{11}(r)\Sigma_{33}(r)-{\Sigma_{13}(r)^2}}\right|}{(2\pi)^{\frac32}\sqrt{\widetilde\Sigma _{11}(r) \widetilde\Sigma _{22}(r)}}\,e^{-\frac12|z|^2}\,dz\, dr\,,
	\end{equation*}
	where
	\begin{align*}
		\Sigma _{11}(r)& :=  4\sum _{l=1}^\infty    l^{4-2s} J_l(r)^2-\frac{4 \left(\sum _{l=1}^\infty   l^{2-2s} J_l(r) J'_l(r)\right){}^2}{ \sum _{l=0}^\infty   l^{-2s} J'_l(r)^2}
		\,,\\
		\Sigma _{13}(r)& :=4\sum _{l=1}^\infty (-1)   l^{2-2s} J_l(r) J''_l(r)+\frac{4 \sum _{l=1}^\infty   l^{2-2s} J_l(r) J'_l(r)  \sum _{l=1}^\infty  l^{-2s} J'_l(r) J''_l(r)}{\sum _{l=1}^\infty  l^{-2s} J'_l(r)^2}
		\,,\\
		\Sigma _{22}(r)& :=  4\sum _{l=1}^\infty    l^{2-2s} J'_l(r)^2-\frac{4 \left(\sum _{l=1}^\infty   l^{2-2s} J_l(r) J'_l(r)\right){}^2}{\sum _{l=1}^\infty    l^{2-2s} J_l(r)^2}
		\,,\\
		\Sigma _{33}(r)& := 4\sum _{l=1}^\infty  l^{-2s} J''_l(r)^2-\frac{4\left( \sum _{l=1}^\infty  l^{-2s} J'_l(r) J''_l(r)\right){}^2}{ \sum _{l=1}^\infty  l^{-2s} J'_l(r)^2}\,.
	\end{align*}
\end{lemma}
\begin{proof}
	As $Du(r,\te)$ is a non-degenerate Gaussian random variable by
        Lemma~\ref{LemmaNonDeg}, the Kac--Rice integral formula in polar coordinates~\cite[Proposition 6.6]{AW09} ensures that
	\begin{equation}\label{KacRicePolarN2}
		\bE\left(N(\nabla u,
                  R)\right)=\int_{B(R)}\bE\big\{|\det D^2
                u(r,\theta)|\;\big|\; Du(r,\theta)=0\big\}\, \rho_{Du(r,\theta)}(0) \,dr\,d\theta
	\end{equation}
	where $\rho_{Du(r,\theta) }:\RR^2\to[0,\infty)$ denotes the probability distribution
        function of the $\RR^2$-valued random variable~$Du(r,\te)$.

Next, let us reduce the computation of the conditional expectation to 	that  of an ordinary expectation by introducing a new random
	variable~$\zeta(r,\te)$. Just like~$D^2u(r,\te)$,        $\zeta(r,\te)$ will take values in the space of $2\times 2$~
        symmetric matrices, which we shall henceforth identify
	with~$\RR^3$ by labeling the matrix components of a symmetric matrix as
	\begin{equation}\label{zeta}
		\zeta=: \left(
		\begin{array}{ccc}
			\zeta_1  & \zeta_2 \\
			\zeta_2& \zeta_3 \\
		\end{array}
		\right)\,.
              \end{equation}

	Specifically, let us set
	\begin{equation}\label{defzeta}
		\zeta(r,\te):= D^2 u(r,\te)-B(r,\te) Du(r,\te)\,,
	\end{equation}
	where the linear operator~$B(r,\te)$ (which we can regard as a
        $3\times 2$~matrix after
	identifying $D^2 u(r,\te)$ with a 3-component vector) is chosen so that the
	covariance matrix of~$Du(r,\te)$ and~$\zeta(r,\te)$ is~0:
	\[
	B(r,\te):=\bE (D^2 u(r,\te)\otimes D u(r,\te))\big[\bE ( Du(r,\te)\otimes Du(r,\te))\big]^{-1}
	\]
	Indeed, one can plug \eqref{defzeta} in the formula for
        $\bE(\zeta(r,\te)\otimes D u(r,\te))$ and check that
	\[
	\bE(\zeta(r,\te)\otimes D u(r,\te))=0\,.
      \]
	As $D u(r,\te)$ and $\zeta(r,\te)$ are {jointly a} Gaussian vector with zero mean, this condition
	ensures that they are independent random variables.
        This enables us to write the above conditional expectation as
	\begin{align*}
	\bE\big\{ |\det D^2 u(r,\te)| \;\big|\; Du(r,\te)=0\big\}&=
          \bE\big\{ |\det [\zeta(r,\te)+ B(r,\te)Du(r,\te)]| \;
          \big|\; D u(r,\te)=0\big\}\\
          &=\bE|\det\zeta(r,\te)|\,.
	\end{align*}

        Let now us compute the covariance matrix of
        $\zeta(r,\te)$. Since the variance matrix of $Du(r,\te)$ is
        independent of~$\te$, let us simply write $\var Du(r)$, and
        similarly with other rotation-invariant quantities. One then has
	\begin{equation}\label{eq:var zeta}
		\var\zeta(r)=\var D^2 u(r)-\cov(D^2 u,D u
                )(r)\cdot\var D u(r)^{-1}\cdot\cov(D^2 u,D u )(r)^\top
	\end{equation}
	Arguing as in Lemma \ref{LemmaNonDeg} and using that we have
        identified $D^2u(r,\te)$ with a 3-component vector, one finds that
        \[
\var D^2u (r):=\bE[D^2u(r,\te)\otimes D^2u(r,\te)]
        \]
        is given by the $3\times 3$ matrix
	\begin{equation*}
		\var D^2u (r)=\left(
		\begin{array}{ccc}
			4 \sum _{l=1}^\infty    l^{4-2s} J_l(r)^2 & 0 & -4 \sum _{l=1}^\infty l^{2-2s} J_l(r) J''_l(r)\\
			0 & 4 \sum _{l=1}^\infty    l^{2-2s} J'_l(r)^2& 0 \\
			-4 \sum _{l=1}^\infty l^{2-2s} J_l(r) J''_l(r) & 0 & 4 \sum _{l=1}^\infty  l^{-2s} J''_l(r)^2\\
		\end{array}
		\right)\,.
	\end{equation*}
	 Similarly,
	\begin{equation}\label{mtx:cov u''u'}
		\cov(D^2 u,D u )(r)=\left(
		\begin{array}{cc}
			0 & -4  \sum _{l=1}^{\infty }  l^{2-2s}   J_l(r) J'_l(r) \\[1mm]
			4  \sum _{l=1}^{\infty }  l^{2-2s}   J_l(r) J'_l(r) & 0 \\[1mm]
			0 & 4  \sum _{l=1}^{\infty } l^{2-2s}   J'_l(r) J''_l(r) \\
		\end{array}
		\right)
	\end{equation}
	Combining these formulas, we derive that
	\begin{equation}\label{eqSi}
		\Si(r):=\var\zeta(r,\te)=\left(
		\begin{array}{ccc}
			\Sigma _{11}(r) & 0 & \Sigma _{13}(r) \\
			0 & \Sigma _{22}(r) & 0 \\
			\Sigma _{13}(r) & 0 & \Sigma _{33}(r) \\
		\end{array}
		\right)\,,
	\end{equation}
	where $\Si_{jk}(r)$ are defined as in the statement of the
        lemma.

	Let us now consider the Cholesky decomposition of this matrix:
	\begin{equation*}
		\Si(r)=M(r)^\top M(r)\,,
              \end{equation*}
              where the matrix~$M(r)$ is given by
              \[
M(r)\coloneqq \left(
		\begin{array}{ccc}
			\sqrt{\Sigma_{11}(r)} & 0 & \frac{\Sigma_{13}(r)}{\sqrt{\Sigma_{11}(r)}} \\
			0 & \sqrt{\Sigma_{22}(r)} & 0 \\
			0 & 0 & \sqrt{\Sigma_{33}(r)-\frac{\Sigma_{13}(r)^2}{\Sigma_{11}(r)}} \\
		\end{array}
		\right)\,.
                \]
                As the matrix $\Sigma(r)$ is positive definite and
                $\zeta(r,\te)$ is a Gaussian random variable with zero
                mean and variance~$\Si(r)$, one then infers that the
                3-component
                random variable
                \[
                  Z(r,\te)\coloneqq \zeta(r,\te)^\top M(r)^{-1}
                \]
                is
                Gaussian, has zero mean and its variance matrix is the
                identity. It is thus straightforward that
	\begin{align*}
		\bE|\det\zeta(r,\te)|&=\int_{\RR^3}\left|y_1y_3-y_2^2 \right|\rho_{\zeta(r,\te)}(y)\,dy\\
		&=\int_{\RR^3}\left| z_1^2 \Sigma_{13}(r)-z_2^2 \Sigma_{22}(r)+z_3 z_1  \sqrt{\Sigma_{11}(r)\Sigma_{33}(r)-{\Sigma_{13}(r)^2}}\right|\frac{e^{-\frac12|z|^2} }{(2\pi)^{\frac32} }\,dz\,,
	\end{align*}
	where
$$\rho_{\zeta(r,\te)}(y):=\frac{\exp\Big(-\frac12y\cdot \Sigma^{-1}y\Big)}{(2\pi)^{3/2}(\det\Sigma(r))^{1/2}}$$
is the probability density
        distribution of the random variable~$\zeta(r,\te)$ and we have used the change of variables
	$$
y_1=: \sqrt{\Sigma_{11}(r)} z_1\,,\qquad y_2=:
\sqrt{\Sigma_{22}(r)}z_2\,,\qquad y_3=: \frac{\Sigma_{13}(r)}{\sqrt{\Sigma_{11}(r)}} z_1+\sqrt{\Sigma_{33}(r)-\frac{\Sigma_{13}(r)^2}{\Sigma_{11}(r)}} z_3\,.
	$$
	and the fact that the Jacobian determinant is $\det M(r)=(\det
        \Sigma(r))^{\frac12}$.
        The lemma follows using that the probabability density
        function of the Gaussian random variable $Du(r,\te)$ is
	\begin{equation}\label{eq:rho u'}
		\rho_{D u(r,\te)}(0)=\frac{1}{2\pi\sqrt{\widetilde\Sigma_{11}(r) \widetilde\Sigma_{22}(r)}}
              \end{equation}
              as a consequence of the formula for $\var D u(r,\te)$
              computed in  Lemma~\ref{LemmaNonDeg} and of the fact
              that the density function of an $\RR^k$-valued Gaussian random variable~$Y$
              with zero mean and variance matrix~$\Si$ is
              \[
\rho_Y(y):=(2\pi)^{-\frac k2} (\det\Si)^{-\frac12}e^{-\frac12 y\cdot\Si^{-1}y}\,.
                \]
              \end{proof}

\subsection{Some technical lemmas}

            In the next subsections, we will discuss the behavior of the formula for the
            expected number of critical points that we have computed
            in Lemma~\ref{L:KR est} above. The analysis will strongly
            depend on the value of the parameter~$s$. In the computations,
            we will use several technical lemmas
            repeatedly, often without further mention.

\begin{lemma} \label{claim:int zeta approx}
Given constants of the form $a_{jk}(r)=\tilde{a}_{jk}(r)+\ep_{jk}(r)$,
with $1\leq j,k\leq m$,
\[
\int_{\bR^m}\bigg|\sum_{1\leq j,k\leq
  m}a_{jk}(r)z_jz_k\bigg|\,e^{-\frac12|z|^2}\,dz=\int_{\bR^m}\bigg|\sum_{1\leq
  j,k\leq m}\tilde a_{jk}(r)z_jz_k\bigg|\,e^{-\frac12|z|^2}\,dz +O\Big(\max_{1\leq j,k\leq m}|\ep_{jk}(r)|\Big)\,.
\]
\end{lemma}
\begin{proof}
  It stems from the elementary estimate
  \[
\left| \bigg|\sum_{1\leq j,k\leq
  m}a_{jk}(r)z_jz_k\bigg|- \bigg|\sum_{1\leq j,k\leq
  m}\tilde a_{jk}(r)z_jz_k\bigg|\right|\lesssim |z|^2\max_{1\leq
j,k\leq m}|\ep_{jk}(r)|\,.
  \]
\end{proof}

\begin{lemma}\label{claim:int little o}
	Let $q:[1,\infty)\to(0,\infty)$ be a continuous function with
        $\int_1^\infty q(r)\, dr=\infty$. Then, for $r\gg1$ and any fixed~$r_0$,
	\begin{equation*}
		{\int_{r_0}^r o(q(r'))\,dr'}=o\left(\int_{r_0}^r q(r')\,dr'\right)\,.
	\end{equation*}
\end{lemma}
\begin{proof}
Consider any $\ep>0$ and assume, without any loss of generality, that $o(q(r'))\geq0$. By definition, there is some $R_\ep$ such that
$o(q(r))\le \varepsilon q(r)$ for all $r>R_\ep$. Now set
$Q(r):=\int_{r_0}^r q(r')\,dr'$ and write
	\begin{align*}
		\frac{\int_{r_0}^r o(q(r'))\, dr'}{Q(r)}&
                                                         =\frac{\int_{r_0}^{R_\varepsilon} o(q(r'))\, dr'}{Q(r)}+\frac{\int_{R_\varepsilon}^r o(q(r'))\, dr'}{Q(r)}\\
          &\le\frac{C_\ep}{Q(r)}+\frac{\ep \int_{R_\varepsilon}^r q(r')\, dr'}{Q(r)}=o(1)+\ep
	\end{align*}
        as $r\to\infty$, since $Q(r)\to\infty$. Letting $\ep\to0$, the
        result follows.
      \end{proof}

      The following lemma will be very useful in the analysis of the
      asymptotic behavior of the number of critical points of~$u$:

      \begin{lemma}\label{L.pi/R}
        Consider a positive smooth $\pi$-periodic function $P$ and constants
        $a\geq0$ and $b\in\RR$. If $a=0$, we also assume that
        $b\geq0$. Then, for $R\gg1$,
        \[
\int_{\pi}^R  r^a(\log r)^b\, P(r)\, dr\sim \frac{ R^{a+1}(\log R)^b}{\pi(a+1)} \int_0^\pi
P(r)\, dr\,.
\]
\end{lemma}

\begin{proof}
Let us define $J:=\lfloor R/\pi\rfloor$ and write $R=J\pi+R_1$, with $0\leq
R_1<\pi$. We can then write
	\begin{align*}
		\int_\pi^R r^a(\log r)^b\, P(r)\, dr=\sum_{j=1}^{ J-1}\int_{\pi j}^{\pi (j+1)} r^a(\log r)^b\, P(r)\, dr
	+\int_{\pi  J}^{\pi J + R_1} r^a(\log r)^b\, P(r)\, dr\,.
	\end{align*}
The second term is obviously bounded as
\[
\left|\int_{\pi  J}^{\pi J + R_1} r^a(\log r)^b\, P(r)\,
  dr\right|\lesssim R^a(\log R)^b
\]

To estimate the first term, let
\[
B:=\int_0^\pi
P(r)\, dr\,.
\]
As the function $r^a(\log r)^b$ is increasing for large enough $r$, we have
\[
B(\pi j)^a[\log (\pi j)]^b  \leq \int_{\pi j}^{\pi (j+1)} r^a(\log r)^b\, P(r)\, dr\leq B[\pi (j+1)]^a[\log (\pi (j+1))]^b
\]
if~$j$ is larger that a certain integer~$J_{a,b}$.
With $\eta=0,1$, we can use the following asymptotic formula, which is an easy consequence of the Euler-Maclaurin formula,
\[
\sum_{j=J_{a,b}}^{J-1}[\pi (j+\eta)]^a[\log (\pi (j+\eta))]^b\sim
\frac{\pi^a (J+\eta-1)^{a+1}[\log (\pi (J+\eta-1))]^b}{a+1}\sim
\frac{R^{a+1}(\log R)^b}{\pi(a+1)}
\]
to derive the formula of the statement. Here we have used that $\pi J=R+O(1)$ and that the integral over $r\in[\pi,\pi J_{a,b}]$ is obviously bounded independently of~$R$.
\end{proof}

Before discussing the behavior of $\bE\crit$ in the different
regularity regimes, one should note that the integral appearing in
Lemma~\ref{L:KR est} is remarkably hard to analyze. We will be able to
obtain much more convenient integral representations by means of the
following lemma:

	\begin{lemma}\label{L.integral}
	Let $A,B,C$ be real constants. Then
		\begin{equation*}
                  \int_{\RR^3}\left|A z_1^2+B
                   z_2^2+ 2 C z_1 z_3\right|
                 \frac{e^{-\frac12|z|^2}}{(2\pi)^{3/2}}\,dz=\frac2\pi
                 \int_0^\infty
                 \frac{1-a(t)\cos{\frac12}\Phi(t)}{t^2}\,dt\,,
               \end{equation*}
               where
               \begin{align*}
                 \Phi(t)&:= \arg \left((1-2 i B t) \left(1-2
                     i A t+4 C^2 t^2\right)\right)\,,\\
                 a(t)&:=(1+4B^2t^2)^{-\frac14}\big[ (1+4C^2t^2)^2+4A^2t^2\big]^{-\frac14}\,.
\end{align*}
	\end{lemma}
        \begin{proof}
	Defining the matrix
	\begin{equation*}
		M\coloneqq {\left(
			\begin{array}{ccc}
				A & 0 & C\\
				0 & B & 0 \\
				C & 0 & 0 \\
			\end{array}
			\right)}\,,
                    \end{equation*}
                    one can write the above integral as
                    \[
Q:=\int_{\RR^3}\left|A z_1^2+B
                   z_2^2+ 2 C z_1 z_3\right|
                 \frac{e^{-\frac12|z|^2}}{(2\pi)^{3/2}}\,dz=\int_{\RR^3}|z\cdot Mz|
                 \frac{e^{-\frac12|z|^2}}{(2\pi)^{3/2}}\,dz\,.
                    \]
The results about Gaussian integrals involving an absolute value
function derived in~\cite[Theorem 2.1]{LW09} therefore ensure that
\[
Q=\frac{2}{\pi} \int_0^{\infty
        } \left[1- \frac{\det(
			I-2 i t M)^{-\frac12}+ \det(
			I+2 i t M)^{-\frac12}}{2}\right]\frac{dt}{t^2} \,.
                    \]
Now a straightforward computation yields the formula in the statement.
\end{proof}

      \subsection{The case $\boldsymbol{ s<\frac12}$}

      We are ready to compute the asymptotics for the number of
      critical points when $s<\frac12$:

      \begin{lemma}\label{L.main1}
        If $s<\frac12$,
        \[
\lim_{R\to\infty} \frac{\bE N(\nabla u, R)}{{R^2}}={\ka(s)}
\]
with
        \begin{equation}\label{ka(s)}
	\ka(s)\coloneqq\frac{1}{{2}}\frac{1}{{\sqrt{ 2-s} }}\int_{\RR^3}\left|\sqrt{\frac{{1-2s}}{{8-4s }}} \left(z_1^2-z_2^2\right)+ z_1 z_3\right| \frac{e^{-\frac12|z|^2}}{(2\pi)^{3/2}}\,dz\,.
\end{equation}
      \end{lemma}

      \begin{proof}
        Let us compute the matrix $\Sigma(r)$. From Equation~\eqref{eqSi} and the
        asymptotic
        formulas for sums of Bessel functions recorded in Corollary~\ref{C.formulas}, it follows that
        \[
\Si(r)=\Si^0(r)+\cR(r)\,,
        \]
        where the leading contribution is
\begin{equation*}
\Sigma^0(r):=\left(	{\everymath={\displaystyle}
	\begin{array}{ccc}
	\frac{2^{2s -3} \Gamma (5-2s) r^{4-2s}}{\Gamma \left(3-s\right)^2} & 0 & \frac{\Gamma \left(\frac{3}{2}-s\right) r^{2-2s}}{\sqrt{\pi } \Gamma \left(3-s\right)} \\
		0 & \frac{\Gamma \left(\frac{3}{2}-s\right) r^{2-2s}}{\sqrt{\pi } \Gamma \left(3-s\right)} & 0 \\
		\frac{\Gamma \left(\frac{3}{2}-s\right) r^{2-2s}}{\sqrt{\pi } \Gamma \left(3-s\right)} & 0 & \frac{3 \Gamma \left(\frac{1}{2}-s\right) r^{-2s }}{2\sqrt{\pi } \Gamma \left(3-s\right)} \\
	\end{array}}
	\right)
      \end{equation*}
      and the error is bounded as
      \[
R_{jk}(r)=o(1) \Si^0_{jk}(r)\,.
\]
Here and in what follows, $o(1)$ denotes a quantity that tends to zero
as $r\to\infty$.

Let us define
\begin{equation}\label{eq:def I}
	I(r,z)\coloneqq\left| z_1^2 \Sigma_{13}(r)-z_2^2 \Sigma_{22}(r)+z_3 z_1  \sqrt{\Sigma_{11}(r)\Sigma_{33}(r)-{\Sigma_{13}(r)^2}}\right|
\end{equation}
and note that, by the formula for~$\Si(r)$ and the asymptotics for
weighted sums of Bessel functions presented in Corollary~\ref{C.formulas},
\begin{align*}
\sqrt{\Sigma_{11}(r)\Sigma_{33}(r)-{\Sigma_{13}(r)^2}}\sim	r^{2-2s} \pi^{-1/4} 2^{s-\frac12 }(2-s)\left(\frac{ \Gamma \left(\frac{1}{2}-s\right) \Gamma (3-2s) }{\Gamma \left(3-s\right)^3}\right)^{1/2}\,.
\end{align*}
Likewise, the quantity
\begin{equation}\label{eq:def D}
	\si(r)\coloneqq {\widetilde\Sigma _{11}(r) \widetilde\Sigma _{22}(r)}
\end{equation}
satisfies the asymptotic bound
\begin{equation*}
	{\si(r)}\sim\frac{2 \Gamma \left(\frac{
              1}{2}-s\right) \Gamma \left(\frac{3}{2}-s\right)
        }{\pi\Gamma \left(2-s\right)^2} r^{2-4s}\,.
\end{equation*}
Finally, the integral
\begin{equation}\label{eq:def cI}
	\cI(r)\coloneqq \frac{1}{{2\pi}\sqrt{  \si(r)}}\int_{\RR^3}{I(z,r)}\frac{e^{-\frac12|z|^2}}{(2\pi)^{3/2}}dz
\end{equation}
can be then estimated, as a consequence of Lemmas~\ref{claim:int zeta
  approx} and~\ref{L.integral} and of the preceding asymptotic bounds, as
\begin{equation*}
	{\cI(r)}\sim \frac{\ka(s)}{{\pi}}\,r\,,
\end{equation*}
where $\ka(s)$ is defined as in the statement.
Thus, the integral formula in Lemma~\ref{L:KR est} ensures that
\begin{equation*}
	\bE\crit\sim 2\int_0^R \ka(s) r\, dr=\ka(s)\,R^2\,.
\end{equation*}
\end{proof}

In the next lemma, we analyze the behavior of the positive constant $\ka(s)$ (which is
written simply as $\ka(s)$ in the statement of Theorem~\ref{T.main}),
for $s<\frac12$. The key idea is to obtain an easier characterization
of this constant as a one-dimensional integral.
Interestingly, the global maximum of~$\ka(s)$ is
attained at $s=0$, that is, in the classical case of random waves with
a translation-invariant covariance kernel. In Figure \ref{F.1} we have plotted $\kappa(s)$ for the first region of $s<1/2$ using the next lemma.

 \begin{figure}[t]
 	\centering
 	\includegraphics[width=0.7\linewidth]{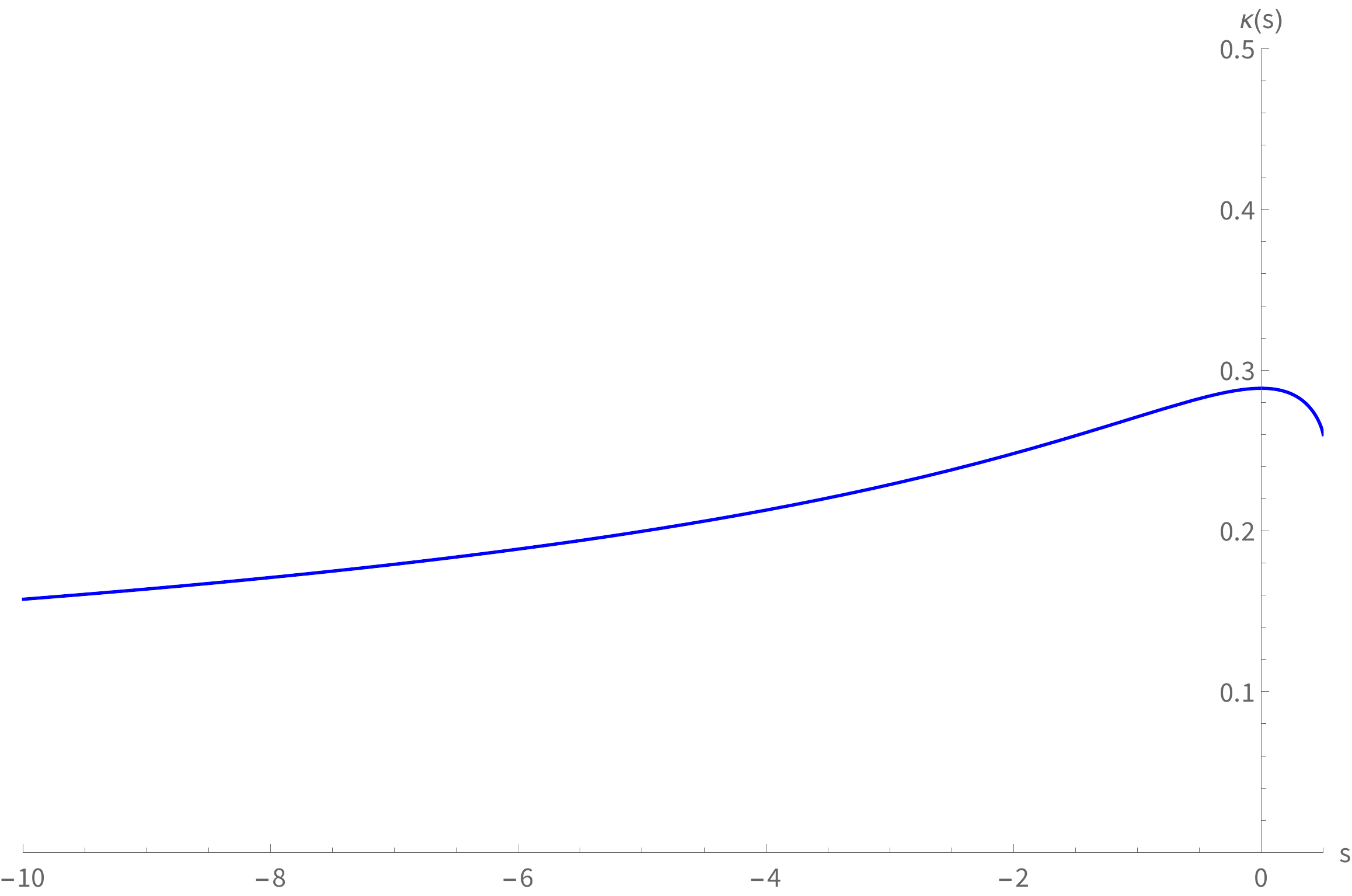}
 	\caption{}
 	\label{F.1}
 \end{figure}

\begin{lemma}\label{prop:cbeta properties}
The function $\ka(s)$ is smooth, strictly  increasing on $ s\in
(-\infty,0)$, and strictly decreasing on $(0,\frac12)$. Furthermore,
\[
\lim_{s\to\frac12^-}\ka(s)=\sqrt{\frac{2}{3}}\frac{{1}}{{\pi}}\,,\qquad \lim_{s\to-\infty}\ka(s)=0\,.
\]
\end{lemma}

\begin{proof}
The limiting values can be computed directly from the formula for
$\ka(s)$. Indeed, the (somewhat surprising) fact that $\ka(s)\to0$ as
$s\to-\infty$ is obvious in view of Equation~\eqref{ka(s)}, and as is the limit
	\begin{equation*}
		\lim_{s\to\frac12^-}\ka(s)=\int_{\RR^3}\frac{{|z_1 z_3|}}{\sqrt{6}  }\frac{e^{-\frac12|z|^2}}{(2\pi)^{3/2}}\,dz=\sqrt{\frac{2}{3}}{\frac{{1}}{\pi }}\,.
	\end{equation*}
	
        To analyze the behavior of~$\ka(s)$ for intermediate values
        of~$s$, we use Lemma~\ref{L.integral} to rewrite~\eqref{ka(s)}
        as
        \[
          \ka(s)={\frac{2}{\pi}}\int_0^\infty \frac{1-a(s,t)\cos{\frac12}\Phi(s,t)}{t^2}\,dt
        \]
        with
        \begin{align*}
a(s,t)&:=\frac{\sqrt{2} (4-2s) }{\big[(1-2s)
    t^6+\left(8 (2-s)^2+6 (1-s)
        t^2\right)^2\big]^{1/4}}\,,\\
          \Phi(s,t)&:=\arg
    \left(4+\frac{2 t^2 \left(-6s +i \sqrt{1-2s}
          t+6\right)}{(4-2s)^2}\right)\,.
        \end{align*}
        Note that
\begin{align*}
	\partial_s a(s,t)&=4s\frac{3   t^2 \left(16 (2-s)^2+t^4+12 (1-s) t^2\right)}{2 \sqrt{2} \left((1-2s) t^6+\left(8 (2-s)^2+6 (1-s) t^2\right)^2\right)^{5/4}}\,,\\
	\partial_s \tan\Phi(s,t)&=-4s\frac{3   t^3 \left(-4s +t^2+8\right)}{2 \sqrt{1-2s} \left(8 (2-s)^2+6 (1-s) t^2\right)^2}\,
\end{align*}
because
$$
\Phi(s,t)=\arctan\left(\frac{\sqrt{1-2s} t^3}{8 (2-s)^2+6 (1-s) t^2}\right)=\arctan \tan\Phi(s,t).
$$
Using that the polynomials appearing on the numerators are all
positive for $t>0$ and $s<\frac12$, it follows that $\ka'(s)/s<0$ for
all $s\in(-\infty,0)\cup(0,\frac12)$. The result then follows.
\end{proof}
\begin{remark}\label{rem:kappa 0}
	In the case $s=0$, where $\ka(s)$ attains its maximum, we recover the well-known asymptotic
	formula ({see Appendix \ref{ApCompTI}}) for the expected number of critical points:
	\begin{equation*}
		\ka(0)=\int_{\RR^3}\frac{\left| z_1^2+2 \sqrt{2} z_3 z_1-z_2^2\right| }{8 }\frac{e^{-\frac12|z|^2}}{(2\pi)^{3/2}}dz=\frac{1}{2\sqrt{3}}=0,2886\dots
	\end{equation*}
where we have used that for $s=0$ the integral above becomes
\begin{equation*}
	\frac{2 }{\pi }\int_0^{\infty } \frac{1-\frac{2}{\sqrt{-\frac{i t^3}{2}+3 t^2+16}}-\frac{2}{\sqrt{\frac{i t^3}{2}+3 t^2+16}}}{t^2} \, dt=\frac{1}{2 \sqrt{3}}.
\end{equation*}
\end{remark}
\subsection{The case $\boldsymbol{ s=\frac12}$}

We shall next show that, in spite of the appearance of logarithmic
terms in the formulas, the asymptotic behavior in the case $s=\frac12$
coincides with the limit as $s\to\frac12^+$ of the formula derived in
Lemma~\ref{L.main1}.

\begin{lemma}\label{L.main2}
  For $s=\frac12$,
  \[
\bE \crit\sim \sqrt{\frac{2}{3}}\frac{{1}}{\pi } R^2\,.
  \]
\end{lemma}

\begin{proof}
From Equation~\eqref{eqSi} and Corollary~\ref{C.formulas}, we infer
that in the case $s=\frac12$, we can write
\[
\Sigma(r)=\Sigma^0(r)+\cR(r)
\]
where
\begin{equation*}
\Sigma^0(r)=	\left(	{\everymath={\displaystyle}
	\begin{array}{ccc}
		\frac{8 r^3}{3} & 0 & \frac{4 r}{3} \\
		0 & \frac{4 r}{3} & 0 \\
		\frac{4 r}{3} & 0 & \frac{4 \log r }{r} \\
	\end{array}}
	\right)
\end{equation*}
and  the error is bounded as
$\cR_{ij}(r)=\Sigma_{ij}^0(r)\,o(1)$. Therefore,
\begin{equation*}
	\sqrt{\Sigma_{11}(r)\Sigma_{33}(r)-{\Sigma_{13}(r)^2}}\sim\frac{4}{3\pi} r \sqrt{6 \log r}\,.
\end{equation*}
Likewise, the function~$\si(r)$ defined in~\eqref{eq:def D} satisfies
\begin{equation*}
	\si(r)\sim \frac{16 \log r}{\pi^2}.
\end{equation*}
Plugging these formulas in~\eqref{eq:def cI}, we obtain
\begin{equation*}
	{\cI(r)}\sim \frac{r\int_{\RR^3}{|z_1
              z_3|}{e^{-\frac12|z|^2}}dz }{\sqrt 6 \pi(2\pi)^{3/2}  }=\sqrt{\frac{2}{3}}\frac{{r}}{\pi ^2}\,.
        \end{equation*}
\end{proof}

\subsection{The case $\boldsymbol{\frac12<s<\frac32}$}

We shall next show that, in the regime $\frac12<s<\frac32$, the
expected number of critical points contained in a large disk also
grows like the area. The associated proportionality constant, which we
denote by $\ka(s)$, turns out to be smooth on
$(-\infty,\frac12)\cup(\frac12,\frac32)$ but only continuous at $s=\frac12$.

\begin{lemma}
  \label{L.main3}
  For $\frac12<s<\frac32$, then $\bE N(\nabla u, R)\sim\ka(s)R^2$ with
  \begin{equation*}
		\ka(s):=\frac{1}{\pi}\sqrt{\frac{3-2s}{4-2s}}\,.
	\end{equation*}	
\end{lemma}

\begin{proof}
By Equation~\eqref{eqSi} and Corollary \ref{C.formulas},
$\Si(r)=\Si^0(r)+\cR(r)$ with
\begin{equation*}
\Sigma^0(r)=	\left( {\everymath={\displaystyle}	
	\begin{array}{ccc}
		\frac{2^{2 s-3} r^{4-2 s} \Gamma (5-2 s)}{\Gamma (3-s)^2} & 0 & \frac{r^{2-2 s} \Gamma \left(\frac{3}{2}-s\right)}{\sqrt{\pi } \Gamma (3-s)} \\
		0 & \frac{r^{2-2 s} \Gamma \left(\frac{3}{2}-s\right)}{\sqrt{\pi } \Gamma (3-s)} & 0 \\
		\frac{r^{2-2 s} \Gamma \left(\frac{3}{2}-s\right)}{\sqrt{\pi } \Gamma (3-s)} & 0 & \frac{4^{2-s} \left(4^s-1\right) \zeta (2 s)}{\pi  r \left(\left(4^s-2\right) \sin (2 r)+4^s\right)} \\
	\end{array}}
	\right)
\end{equation*}
and $\cR_{ij}=\Sigma_{ij}^0(r)\,o(1)$. Therefore, as ${4-4s}<{3-2s}$,
\begin{equation*}
	\sqrt{\Sigma_{11}(r)\Sigma_{33}(r)-{\Sigma_{13}(r)^2}}\sim\sqrt{\frac{2}{\pi }} \sqrt{\frac{\left(4^s-1\right) r^{3-2 s} \zeta (2 s) \Gamma (5-2 s)}{\Gamma (3-s)^2 \left(\left(4^s-2\right) \sin (2 r)+4^s\right)}}\,.
\end{equation*}
Similarly, and using the same notation as in the last two subsections,
\begin{equation*}
	\si(r)\sim\frac{4 r^{1-2 s} \zeta (2 s) \Gamma (3-2 s) \left(\left(4^s-2\right) \sin (2 r)+4^s\right)}{\pi  \Gamma (2-s)^2}\,.
\end{equation*}
One can then plug these formulas in~\eqref{eq:def cI} to find
\begin{equation*}
	{\cI(r)}\sim\frac{r}{\pi\left(1+\left(1-2^{1-2s}\right) \sin2r\right)}\sqrt{\frac{2^{-2s } \left(1-2^{-2s }\right) (3-2s)}{(4-2s) }}\int_{\RR^3}{|z_1z_3|}\frac{e^{-\frac12|z|^2}}{(2\pi)^{3/2}}dz\,.
\end{equation*}
As $2^{1-2s}<1$, this immediately implies
\begin{equation*}
	\bE N(\nabla u, R)\sim \frac4\pi \sqrt{\frac{2^{-2s } \left(1-2^{-2s }\right) (3-2s)}{(4-2s) }}\int_0^R \frac{r}{1+(1-2^{1-2s}) \sin 2 r}\,dr\,.
      \end{equation*}
      As
\begin{equation}\label{intb}
	\int_0^{\pi } \frac{1}{1+b \sin2r} \, dr=\frac{\pi }{\sqrt{1-b^2}}
      \end{equation}
for all $|b|<1$, the formula of the statement now follows using Lemma~\ref{L.pi/R}.
\end{proof}

\begin{remark}
It follows from Lemmas~\ref{L.main1}, \ref{L.main2} and~\ref{L.main3}
that $\ka(s)\in C^\infty((-\infty,\frac12)\cup(\frac12,\frac32])$, and
that $\ka(s)$ is Lipschitz at $s=\frac12$ but not~$C^1$. It also
follows that
\[
\lim_{s\to-\infty}\ka(s)=\lim_{s\to\frac32^-}\ka(s)=0\,.
  \]
\end{remark}

\subsection{The case $\boldsymbol{ s=\frac32}$}

Here we shall see that the expected number of critical points contained in
a ball of large radius does not grow like the area of the ball any longer:

\begin{lemma}\label{L.main4}
  If $s=\frac32$,
  \[
    \bE N(\nabla u, R)\sim {\frac1\pi}\frac{R^2}{\sqrt{\log R}}\,.
\]
\end{lemma}

\begin{proof}
The argument is essentially as before. Using
Corollary~\ref{C.formulas} and Equation~\eqref{eqSi}, one can write
$\Si(r)=\si^0(r)+\cR(r)$, with
\begin{equation*}
	\Sigma^0(r):={\frac1\pi}	\left(	{\everymath={\displaystyle}
	\begin{array}{ccc}
		4 r & 0 & \frac{4 \log r}{r} \\
		0 & \frac{4 \log r}{r} & 0 \\
		\frac{4 \log r}{r} & 0 & \frac{7 \zeta (3)}{4 r+3 r
                                         \sin 2r} \\
	\end{array}}
	\right)
\end{equation*}
and $\cR_{ij}=\Sigma_{ij}^0(r)\,o(r^0)$. Hence, keeping track of the
errors using Lemmas~\ref{claim:int zeta approx}-\ref{claim:int little o} as before,
\begin{align*}
	\sqrt{\Sigma_{11}(r)\Sigma_{33}(r)-{\Sigma_{13}(r)^2}}&\sim\frac{2}{\pi}\sqrt{\frac{7
                                                                \zeta
                                                                (3)}{3
                                                                \sin2r+4}}\,,\\
  \si(r)&\sim\frac{4 \zeta (3) \log r (3 \sin2r+4)}{\pi^2r^2}\,.
\end{align*}
This readily implies
\[
{\cI(r)}\sim \frac r{\sqrt{\log
    r}}\frac{\sqrt{7} }{ \pi^2\left(3 \sin 2r +4\right) }\,,
\]
so Lemma~\ref{L:KR est} ensures that the expected number of critical
points satisfies
\[
\bE N(\nabla u, R)\sim {\frac{{2}\sqrt{7}}{{\pi^2}} } \int_{{\pi}}^R \frac1{4+3 \sin2r}\frac{r}{\sqrt{\log r}}dr\,.
\]
The asymptotic behavior of this integral is
\begin{align*}
  \int_{\pi}^R \frac1{4+3 \sin2r}\frac{r}{\sqrt{\log
  r}}dr&
  \sim\frac{R^2}{2\pi\sqrt{\log R}}\int_0^\pi \frac1{4+3 \sin2r}dr=\frac{R^2}{2\sqrt{{7}\log R}}\,.
\end{align*}
by Lemma~\ref{L.pi/R}, so the result follows.
\end{proof}

\subsection{The case $\boldsymbol{\frac32<s<\frac52}$}

The analysis of the large~$R$ asymptotics presents no new
difficulties:

\begin{lemma}\label{L.main5}
  For $\frac32<s<\frac52$, $
\bE \crit\sim \ka(s) R^{\frac72-s}$
  with
  \[
\ka(s):=-\frac{2^{2 s+\frac{1}{2}} r^{\frac{5}{2}-s} \sqrt{\frac{\left(4^s-1\right) \Gamma (5-2 s)}{\zeta (2 s-2)}}}{\pi ^{3/2} (7-2 s) \Gamma (3-s)} \int_0^\pi
\frac{dr}{\left(\left(4^s-2\right) \sin (2 r)+4^s\right) \sqrt{4^s-\left(4^s-8\right) \sin (2 r)}}\,.
    \]
{See Figure \textnormal{\ref{F.2}}}.
  \end{lemma}
 \begin{figure}[t]
	\centering
	\includegraphics[width=0.7\linewidth]{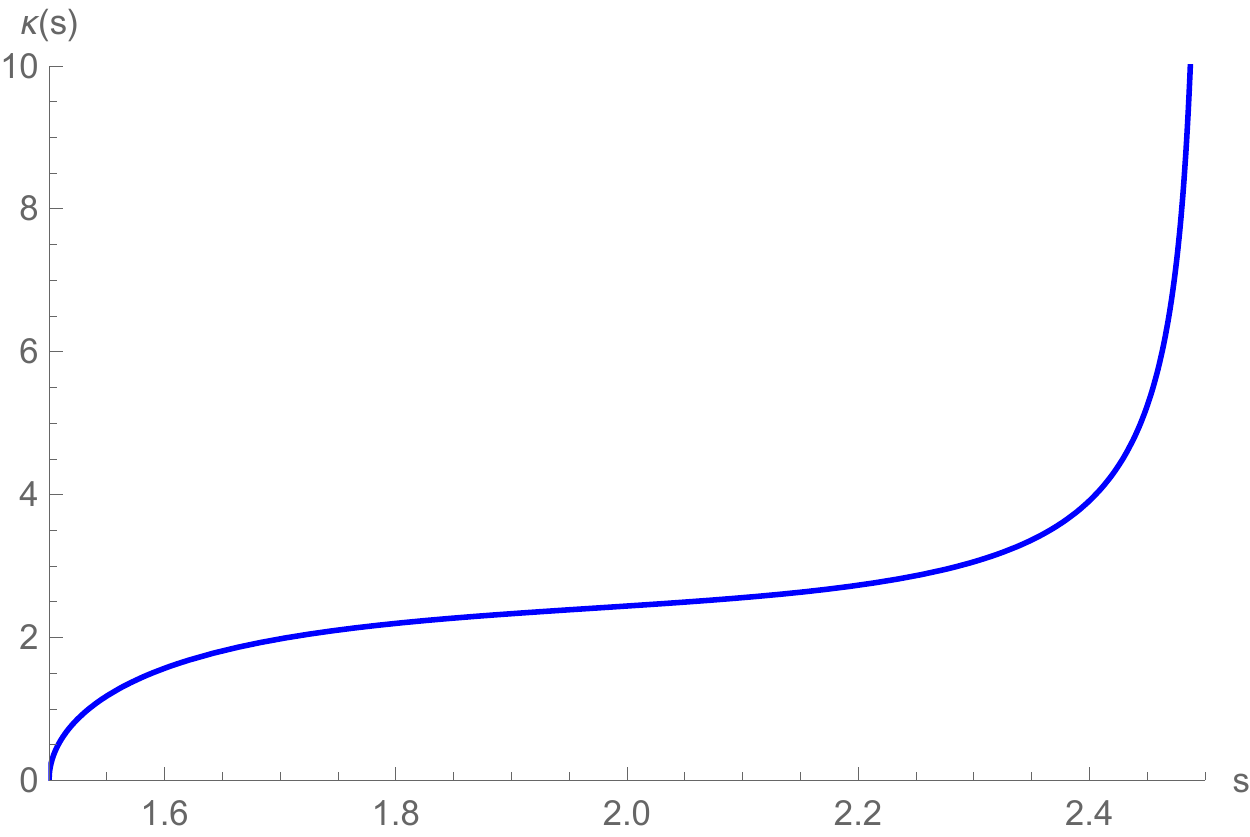}
	\caption{}
	\label{F.2}
\end{figure}
  \begin{proof}
Arguing as before, one finds that $\Si(r)=\Si^0(r)+\cR(r)$ with
\begin{equation*}
	\Sigma^0(r)={\frac{1}{\pi}}	\left(
	\begin{array}{ccc}
		\frac{\pi  2^{2s -3} \Gamma (5-2s) r^{4-2s}}{\Gamma \left(3-s\right)^2} & 0 & \frac{2^{3-2s} \zeta (2s -2) \left(2^{3-2s}-3 \sin2r-5\right)}{r \left(\left(2^{1-2s}-1\right) \sin2r-1\right)} \\
		0 & -\frac{2^{6-2s} \left(2^{2-2s}-1\right) \zeta (2s -2)}{\left(2^{3-2s}-1\right) r \sin2r+r} & 0 \\
		\frac{2^{3-2s} \zeta (2s -2) \left(2^{3-2s}-3 \sin2r-5\right)}{r \left(\left(2^{1-2s}-1\right) \sin2r-1\right)} & 0 & \frac{2^{4-2s} \left(2^{-2s }-1\right) \zeta (2s )}{r \left(\left(2^{1-2s}-1\right) \sin2r-1\right)}
	\end{array}
	\right)
      \end{equation*}
      and $\cR_{ij}=\Sigma_{ij}^0(r)o(1)$. This readily leads to the expression
\begin{equation*}
	{\cI(r)}\sim\frac{\left(2^{2 s-\frac{1}{2}} r^{\frac{5}{2}-s}\right) }{\sqrt{\pi } \Gamma (3-s) \left(\left(4^s-2\right) \sin (2 r)+4^s\right)}\sqrt{\frac{\left(4^s-1\right) \Gamma (5-2 s)}{\zeta (2 s-2) \left(4^s-\left(4^s-8\right) \sin (2 r)\right)}}\,,
      \end{equation*}
      which implies
      \begin{align*}
      	  \bE \crit
       &\sim \frac{4\left(2^{2 s-\frac{1}{2}} \right) }{\sqrt{\pi } \Gamma (3-s) }\sqrt{\frac{\left(4^s-1\right) \Gamma (5-2 s)}{\zeta (2 s-2) }}\times\\
       &\times\int_0^R\frac{r^{\frac{5}{2}-s}}{\left(\left(4^s-2\right) \sin (2 r)+4^s\right)\left(4^s-\left(4^s-8\right) \sin (2 r)\right)}dr
       \,.
      \end{align*}
Applying Lemma~\ref{L.pi/R} once again, one obtains the desired formula.
    \end{proof}

    \subsection{The case $\boldsymbol{s=\frac52}$}

The next lemma shows that at this regularity level, there is another
transition in the asymptotic behavior of the expected number of
critical points of~$u$:

\begin{lemma}\label{L.main6}
If $s=\frac52$, $\bE \crit\sim \tka_{\frac52} R\sqrt{\log R}$
with
\[
\tka_{\frac52}:= \frac4{\pi^2}\sqrt{\frac{31}{\zeta(3)}}\int_0^\pi\frac{dr}{(16+15\sin2r)\sqrt{4-3\sin2r}}{\approx 0.497339}\,.
  \]
\end{lemma}

\begin{proof}
  Arguing as before, one find that $\Si(r)=\Si^0(r)+\cR(r)$ with
  \begin{equation*}
	\Sigma^0(r)={\frac{1}{\pi}}	\left(	{\everymath={\displaystyle}
	\begin{array}{ccc}
		\frac{4 \log r}{r} & 0 & \frac{\zeta (3) (12 \sin2r+19)}{r (15\sin2r+16)} \\
		0 & \frac{7 \zeta (3)}{4 r-3 r \sin 2r} & 0 \\
		\frac{\zeta (3) (12 \sin2r+19)}{r (15 \sin2r +16)} & 0 & \frac{31 \zeta (5)}{64 r+60 r \sin2r} \\
	\end{array}}
	\right)
      \end{equation*}
      and $\cR_{ij}(r)=\Si^0_{ij}(r)\,o(1)$. This eventually yields the
      asymptotic formula
      \[
\cI(r)\sim \frac2{\pi^2}\sqrt{\frac{31}{\zeta(3)}}\frac{\sqrt{\log r}}{(16+15\sin2r)\sqrt{4-3\sin2r}}\,,
        \]
        which implies
        \[
\bE \crit\sim \frac4{\pi}\sqrt{\frac{31}{\zeta(3)}}\int_0^R \frac{\sqrt{\log r}}{(16+15\sin2r)\sqrt{4-3\sin2r}}dr
        \]
        by Lemmas~\ref{L:KR est} and~\ref{claim:int little
          o}. Lemma~\ref{L.pi/R} then yields the desired asymptotic behavior.
\end{proof}

\subsection{The case $\boldsymbol{s>\frac52}$}

In this regime, the proof goes as before, showing that
the expected number of critical points contained in a large ball grows asymptotically
like the radius. However, the explicit formulas
one obtains for the proportionality constant are extremely cumbersome.

\begin{lemma}\label{L.main7}
  For $s>\frac52$, there exists an explicit constant $\ka(s)>0$ such that
  \[
\bE\crit\sim \ka(s)R\,.
    \]
  \end{lemma}

  \begin{proof}
As in the previous cases, let us write $\Si(r)=\Si^0(r)+\cR$ with
$\cR_{ij}=\Si^0(R)\,o(1)$ and
\begin{equation*}
	\Sigma^0(r)=	\frac{1}{{\pi}r}\left(
	\begin{array}{ccc}
		\Sigma_{11}(r) & 0 & \frac{2^{3-2s} \zeta (2s -2) \left(2^{3-2s}-3 \sin2r-5\right)}{\left(2^{1-2s}-1\right) \sin2r-1} \\
		0 & -\frac{2^{6-2s} \left(2^{2-2s}-1\right) \zeta (2s -2)}{\left(2^{3-2s}-1\right)  \sin2r+1} & 0 \\
		\frac{2^{3-2s} \zeta (2s -2) \left(2^{3-2s}-3 \sin2r-5\right)}{\left(2^{1-2s}-1\right) \sin2r-1} & 0 & \frac{2^{4-2s} \left(2^{-2s }-1\right) \zeta (2s )}{\left(2^{1-2s}-1\right) \sin2r-1} \\
	\end{array}
	\right)\,.
\end{equation*}
Here
$$
\Sigma_{11}(r)\coloneqq4 \zeta (2s -4) \left(\left(2^{5-2s}-1\right) \sin2r+1\right)+\frac{4\left(2^{3-2s}-1\right)^2 \cos ^2(2 r) \zeta (2s-2)^2}{\zeta (2s ) \left(\left(2^{1-2s}-1\right) \sin2r-1\right)}\,.
$$
Note that all the nonzero matrix
components are exactly of order~$1/r$. While this fact does not make the
problem any harder from a conceptual point of view, it leads to
cumbersome expressions for the various quantities appearing in the
equations.

Specifically, it is not hard to show that
\begin{equation*}
	\si(r)\sim-\frac{16 \zeta (2s -2) \zeta (2s ) \left(\left(2^{1-2s}-1\right) \sin2r-1\right) \left(\left(2^{3-2s}-1\right) \sin2r+1\right)}{r^2}\,.
      \end{equation*}
      Plugging this formula in the expression for $I(r,z)$, one finds that
      \[
\cI(r)\sim \int_{\RR^3}|Az_1^2+Bz_2^2{+}2Cz_1z_2|\frac{e^{-\frac12|z|^2}}{(2\pi)^{3/2}}\,dz\,,
\]
where the constants
\begin{align*}
\al&:=\frac1\pi\Big[ \ze(2s-2)\ze(2s)[1{+}(1-2^{1-2s})\sin 2r][1+(2^{3-2s}-1)\sin2r]\Big]^{-\frac12}\\
A&:= \al 2^{-2s}\ze(2s-2)\frac{5{-}2^{3-2s}{+}3\sin2r}{1-(1-2^{1-2s})\sin
   2r}\,,\\
  B&:= \al
     2^{3-2s}\ze(2s-2)\frac{2^{2-2s}-1}{1+(2^{3-2s}-1)\sin2r}\,,\\
  C&:= \frac{\al 2^{-s-1}}{1{+}(1-2^{1-2s})\sin 2r}\times\\
  &\quad\times\Big[(1-2^{-2s})\ze(2s-4)\ze(2s) [1{+}(1-2^{1-2s})\sin 2r][1+({-1+}2^{5-2s})\sin2r] \\
     &\quad\quad +\ze(2s-2)^2\big[ {-}(1-2^{-2s})(1-2^{3-2s})^2 \cos^22r
     -2^{-2s}(2^{3-2s} -3\sin2r -5)^5\big]^2\Big]^{\frac12}
\end{align*}
are smooth functions of $\sin 2r$.

Lemma~\ref{L.integral} then shows that
\[
\cI(r)\sim F(s,\sin 2r)
\]
for some explicit smooth function of the form
\[
F(s,\sin 2r)={\frac{2}{\pi}}\int_0^\infty\frac{1-a(t,s,\sin
  2r)\cos\frac{1}{2}\Phi(t,s,\sin 2r)}{t^2}dt\,.
\]
Since
\[
  a(t,s,\sin 2r)= \Big[(1+4B^2t^2)\big[ (1+4C^2t^2)^2+4A^2t^2\big]\Big]^{-\frac14}<1
\]
for all $r$ and all $t>0$, it stems that
\[
  F(s,\sin 2r){>0}\,.
\]
Lemmas~\ref{L:KR est}, \ref{claim:int little o} and~\ref{L.pi/R} then ensure that
\[
\bE\crit\sim \ka(s) R
\]
with
\[
\ka(s):={2 } \int_0^\pi F(s,\sin 2r)\,dr\,.
\]
\end{proof}

One can now read the asymptotic behavior of $\bE\crit$ in any
regularity regime from the lemmas that we have established in this
section. Theorem~\ref{T.main} is therefore proven.

\section{Asymptotics for the number of critical points in the high
  regularity case}\label{ap:smooth case}

This section is devoted to the proof of Theorem~\ref{T.as}. As all along this paper, we shall take the
definition~\eqref{defu2} for the Gaussian random function~$u$.


\subsection{Some non-probabilistic lemmas}

Before presenting the proof of this theorem, we need to prove a few
auxiliary results that do not use the fact that $u$ and~$f$ are random
functions. Specifically, these lemmas concern solutions to the
Helmholtz equation on~$\RR^2$ of the form
\[
v(x):=\int_{\TT}e^{-i x\cdot E(\phi)}\, g(\phi)\, d\phi
\]
where $g\in H^m(\TT)$ for a certain real~$m$ and the standard embedding $E:\TT\to\RR^2$ is given by~\eqref{defE}.

We start by recalling the following result on the asymptotic
behavior of~$v$, which we proved in~\cite[Proposition 2.2 and Remark~3.2]{EPR20}. In what follows, we will
denote the real and imaginary parts of a function~$g$ by $g\R$ and $g\I$, respectively.

\begin{lemma}\label{L.stationary}
	If $m>9/2$, for $r\gg1$ one has
	\begin{align*}
		v&=
		\bigg(\frac{8\pi}{r}\bigg)^{\frac12}\big[
		g\I(\te)\, \sin(r-\tfrac{\pi}{4})+ g\R(\te)\,\cos(r-\tfrac{\pi}{4})+\cR_1\big]\,,\\
		\pd_r v&=
		\bigg(\frac{8\pi}{r}\bigg)^{\frac12}\big[
		g\I(\te)\, \cos(r-\tfrac{\pi}{4}) -g\R(\te)\,\sin(r-\tfrac{\pi}{4})+\cR_2\big]\,,\\
		\pd_\te v &=
		\bigg(\frac{8\pi}{r}\bigg)^{\frac12}\big[
		 g\I'(\te)\, \sin(r-\tfrac{\pi}{4})+ 	g\R'(\te)\,\cos(r-\tfrac{\pi}{4})+\cR_3\big]\,,
	\end{align*}
	where
	the errors are bounded as
	\[
	|\cR_1|+ |\nabla\cR_1|+ |\nabla^2\cR_1|+ |\cR_2|
	+ |\cR_3|\lesssim \frac{1}{r}\,.
	\]
      \end{lemma}

The following theorem provides very precise asymptotic information about the
critical points of~$v$:

\begin{lemma}\label{prop: critical points}
Assume that $m>9/2$, that $g$ does not vanish on~$\TT$, and that all the
critical points of $|g|$ are non-degenerate. If $\phi^*$ is a critical
point of~$|g|$, then for each large enough positive integer~$n$ there
exists a critical point~$(r_n^*,\te_n^*)$ of~$v$ such that
\[
|\phi^*-\te^*_n|+ \big|\pi n+ \tfrac{\pi}{4}+\arg g(\phi^*)-r^*_n\big|\lesssim\frac1n\,.
\]
Conversely, if $(r^*,\te^*)$ is a critical point of~$v$, there is some
critical point~$\phi^*$ of~$|g|$ such that
\[
|\phi^*-\te^*|\lesssim\frac1{r^*}\,.
\]
\end{lemma}

\begin{proof}
  Let us consider the function
  \[
    V\coloneqq \Real \big[g(\theta)e^{-i(r-\tfrac{\pi}{4})}\big]=g\I(\te)\, \sin(r-\tfrac{\pi}{4})+ g\R(\te)\,\cos(r-\tfrac{\pi}{4})\,,
  \]
  whose critical points $(r^*,\te^*)$
  are the solutions to the equations
\[
\Imag \big [g(\theta^*)e^{-i(r^*-\tfrac{\pi}{4})}\big]=0\,,\qquad \Real \big [g'(\theta^*)e^{-i(r^*-\tfrac{\pi}{4})}\big]=0\,.
\]
Writing $g=|g|
  e^{i\arg g}$, an elementary calculation shows that $(r^*,\te^*)$ is a critical
  point of~$V$ if and only if
  $r^*= \arg g(\te^*)+\frac\pi4+\pi n$ for
  some integer~$n$ and $\Real
  [\overline{g(\te^*)}g'(\te^*) \big]=0$. As $g$ does not vanish
  on~$\TT$, the latter condition
  simply means that $\te^*$ is a critical point of~$|g|$. Furthermore,
  the Hessian of~$V$ at the critical points is
  \[
D^2V(r^*,\te^*)=(-1)^n \left(
\begin{array}{cc}
	-|g(\te^*)| & |g(\te^*)| (\arg g)'(\te^*)\\
	|g(\te^*)| (\arg g)'(\te^*)& |g|''(\te^*)-|g(\te^*)|[(\arg g)'(\te^*)]^2\\
\end{array}
\right)\,.
  \]
  Therefore,
  \begin{equation}\label{HessV}
\det D^2V(r^*,\te^*)= - |g(\te^*)|\,|g|''(\te^*)\neq0
\end{equation}
because the critical points of~$|g|$ are, by hypothesis, nondegenerate.

  Let us now consider the function
  \[
F(r,\te):=DV(r,\te)- \bigg(\frac{r}{8\pi}\bigg)^{\frac12} Dv(r,\te)\,,
  \]
  where $DV:=(\pd_rV,\pd_\te V)$. Lemma~\ref{L.stationary} ensures
  that
  \[
|F(r,\te)|+ |D F(r,\te)|\lesssim\frac1r\,.
  \]
  As the critical points of~$V$ are uniformly non-degenerate
  by~\eqref{HessV}, Thom's isotopy theorem (as stated, e.g.,
  in~\cite{EPS13}) ensures that $v$ has a critical point at a distance
  at most $C/n$ to each of the critical points $(r^*,\te^*)$ of~$V$ as
  described above, provided that~$n$ is large enough. Furthermore, the asymptotic
  formulas for $Dv$ presented in Lemma~\ref{L.stationary}
  guarantee that all critical points of~$v$ that are far enough from
  the origin must be of this form. The lemma is then proven.
\end{proof}

\subsection{Proof of Theorem~\ref{T.as}}

As $s>5$, Proposition~\ref{P.reg} ensures that $f\in H^{s'}(\TT)$
almost surely for
some $s'>\frac92$. Therefore, if one can prove that, with probability~1, $f$
does not vanish on~$\TT$ and all the critical points of~$|f|$ are nondegenerate, Theorem~\ref{T.as} will follow as an easy
consequence of Lemma~\ref{prop: critical points}.

Proving the first part of this assertion is completely standard, but
the second part is quite harder. In both cases, the proof relies on Bulinskaya's lemma, which one can state as follows~\cite[Proposition 6.11]{AW09}:

\begin{lemma}[Bulinskaya]\label{L.Bulin}
Let $Y :\TT\to\RR^2$ be a random function that is of class $C^1(\TT)$
almost surely. For each $\phi\in\TT$, assume that the random variable
$Y(\phi)$ has a probability density
$\rho_{Y(\phi)}:\RR^2\to[0,\infty)$ that is bounded in some fixed
neighborhood of the origin. Then
\[
\bP\{Y(\phi)=0\; \text{ for some } \phi\in\TT\}=0\,.
\]
\end{lemma}

Armed with Bulinskaya's lemma, it is easy to show that, almost surely, $f$ does not vanish:

\begin{lemma}
With probability~$1$, $f$ does not vanish on~$\TT$.
\end{lemma}

\begin{proof}
By the definition of~$u$, cf.~Equations~\eqref{defu2} and~\eqref{deff}, $\tilde Y(\phi):=(f\R(\phi),f\I(\phi))$ is a Gaussian random field
$\tilde Y:\TT\to\RR^2$ with zero mean. The covariance of~$\tilde Y(\phi)$ can
be computed just as in Lemma~\ref{LemmaNonDeg}, obtaining the
nondegenerate matrix
\begin{align*}
  \var \tilde Y(\phi)&=\bE[\tilde Y(\phi)\otimes \tilde Y(\phi)]= \left(
                                                                    \begin{array}{cc}
                                                                      \pi^{-2}\sum_{l>0, \text{even}}l^{-2s} & 0 \\
                                                                      0 & \pi^{-2}\sum_{l>0, \text{odd}}l^{-2s} \\
                                                                    \end{array}
                                                                  \right)=:\Sigma\,.
\end{align*}
Therefore, $\tilde Y(\phi)$ has a bounded probability density
function
\[
  \rho_{\tilde Y(\phi)}(y):= \frac{\exp\Big({-\frac12y\cdot \Sigma^{-1}y}\Big)}{2\pi (\det \Sigma)^{1/2}}
\]
on~$\RR^2$ because $\Sigma$ is a nondegenerate matrix. Lemma~\ref{L.Bulin} then ensures that $\tilde Y$ does not vanish with
probability~1. As the zeros of~$\tilde Y$ and~$f$ obviously coincide, the
lemma follows.
\end{proof}

The crux of the proof of Theorem~\ref{T.as} is to show that the
critical points of~$|f|$ are nondegenerate. This is not direct because
$|f|$ is {not}\/ a Gaussian variable, and showing that it has a
bounded probability density requires some work. The main
ingredient of the proof is the estimate we present in the following
lemma. The proof is somewhat involved, so we have relegated it to the
next subsection in order to streamline the presentation of the proof
of Theorem~\ref{T.as}. To state the auxiliary result, we will
write points in~$\RR^6$ as
\[
z=(z',z'')\in\RR^4\times\RR^2
\]
with $z':=(z_1,z_2,z_3,z_4)$ and $z'':=(z_5,z_6)$.

\begin{lemma}\label{L.intQ}
  Consider the nonnegative rational function on~$\RR^6$ given by
  \begin{align}\label{defQz}
Q(z):= |z'|^2+ \frac{(z_5-z_1z_3)^2}{z_2^2}+\frac{[(z_5-z_1z_3)^2+z_2^2(z_1z_4+z_3^2-z_6)]^2}{z_2^6}\,.
  \end{align}
  For any constant $c>0$,
  \[
\sup_{|z''|<\frac12}\int_{\RR^4}\frac{e^{-c \,Q(z)}}{z_2^2}\,dz' <\infty\,.
\]
\end{lemma}

Assuming for the moment that this technical lemma holds, proving that
the critical points of~$|f|$ are nondegenerate almost surely is straightforward:

\begin{lemma}
With probability~$1$, all the critical points of $|f|$ are nondegenerate.
\end{lemma}

\begin{proof}
  Let us start by noting that
  \[
|f|\,|f|'=\tfrac12(|f|^2)'= \Real \overline f\,f'= f\R f\R'+f\I f\I'\,.
  \]
  Differentiating this identity, we obtain
  \[
|f|\,|f|''+(|f|')^2=\Real \overline f\, f''+|f'|^2=  f\R f\R''+f\I
f\I''+(f\R')^2+(f\I')^2\,.
  \]
  Therefore, all the critical points of $|f|$ are nondegenerate if and
  only if
  \[
Y:=(f\R f\R'+f\I f\I',  f\R f\R''+f\I
f\I''+(f\R')^2+(f\I')^2):\TT\to\RR^2
  \]
  does not vanish.

  As $Y\in C^2(\TT)$ almost surely because $s>5$, in order to apply
  Bulinskaya's lemma we only need to show that $Y(\phi)$ has a
  probability density that is bounded in a neighborhood of the
  origin. The random variable $Y(\phi)$ is obviously not Gaussian, so
  in order to compute its density we need to argue in an indirect
  way.

  The starting point is the fact that the 2-jet of~$f$,
  \[
Z:=(f\R,f\I,f\R',f\I',f\R'',f\I'')\,,
  \]
  defines a Gaussian random variable $Z:\TT\to\RR^6$ with zero
  mean. Its variance
  \[
\var Z(\phi):=\bE[Z(\phi)\otimes Z(\phi)]\,,
  \]
which does not depend on~$\phi$,  can be computed from the definition
  \[
f(\phi):= \frac1{2\pi}\sum_{l\neq0} i^l a_l |l|^{-s} e^{i l\phi}
  \]
  by arguing just as in the proof of Lemma~\ref{LemmaNonDeg}. It turns
  out that $\var Z(\phi)=\Si$, where $\Si$ is the $6\times 6$ matrix
 \begin{equation*}
\Si:=\left(
	\begin{array}{cccccc}
		a_0 & 0 & 0 & 0 & -b_0 & 0 \\
		0 & a_1 & 0 & 0 & 0 & -b_1 \\
		0 & 0 & b_0 & 0 & 0 & 0 \\
		0 & 0 & 0 & b_1 & 0 & 0 \\
		-b_0 & 0 & 0 & 0 & c_0 & 0 \\
		0 & -b_1 & 0 & 0 & 0 & c_1 \\
	\end{array}
	\right)\,,
\end{equation*}
where
\begin{align*}
  a_i\coloneqq \pi^{-2}\sum_{m=0}^\infty  \sigma_{i+2m}^2\,, \qquad
  b_i\coloneqq \pi^{-2}\sum_{m=0}^\infty  \sigma_{i+2m}^2 ({i+2m})^2\,,\qquad c_i\coloneqq \pi^{-2}\sum_{m=0}^\infty  \sigma_{i+2m}^2 ({i+2m})^4
\end{align*}
and we have set $\si_l:=|l|^{-s}$ for $l\neq0$ and $\si_0:=0$. We have chosen to
write this formula in terms of~$\si_l$ so that it is apparent that the
result only uses the asymptotic properties of the
sequence~$\si_l$. Note that these sums are all convergent because
$s>5$.

The determinant of~$\Si$ is
\begin{align*}
\det \Si&= b_0 b_1 \left(b_0^2-a_0 c_0\right) \left(b_1^2-a_1 c_1\right)\,.
\end{align*}
As $a_ic_i>b_i^2$ strictly by the Cauchy--Schwartz inequality, the
matrix~$\Si$ is invertible.
Therefore, for each~$\phi\in\TT$, the probability density distribution
of~$Z(\phi)$ is given by
the Gaussian function
\[
g(z):=(2\pi)^{-3}(\det\Si)^{-\frac12} e^{-\frac12 z\cdot \Si^{-1} z}
\in C^\infty(\RR^6)\,.
\]

Consider now the map $H:\RR^6\to\RR^6$ given by
\begin{align}\label{Hz}
H(z)&\coloneqq\left(z_{1},z_{2},z_{3},z_{5},z_{1} z_{3}+z_{2} z_{4},z_{1} z_{5}+z_{2} z_{6}+z_{3}^2+z_{4}^2\right)\,.
\end{align}
This map is invertible outside the hyperplane $\{z_2=0\}$, with
inverse
\[
H^{-1}(z):= \left(z_{1},z_{2},z_{3},\frac{z_{5}-z_{1} z_{3}}{z_{2}},z_{4},-\frac{(z_{5}-z_{1} z_{3})^2}{z_{2}^3}-\frac{z_{1} z_{4}+z_{3}^2-z_{6}}{z_{2}}\right)\,,
\]
and its corresponding Jacobian determinant is $\det \nabla
H^{-1}(z)=-{z_{2}^{-2}}$. Therefore, the probability density
distribution of the random variable $H[Z(\phi)]$ is obtained by
pulling back with the map~$H$ the probability distribution of~$Z(\phi)$:
\begin{equation}\label{rhoHZ}
\rho_{H[Z(\phi)]}(z)=|\det \nabla
H^{-1}(z)| \, g[H^{-1}(z)]= (2\pi)^{-3}(\det\Si)^{-\frac12} z_2^{-2} e^{-Q_H(z)}\,.
\end{equation}
with $Q_H(z):= \frac12 H^{-1}(z)\cdot \Si^{-1} H^{-1}(z)$.

Now let $\tH:\RR^6\to\RR^2$ denote the last two components of the
map~\eqref{Hz}, that is,
\[
\tH(z):=\left(z_{1} z_{3}+z_{2} z_{4},z_{1} z_{5}+z_{2} z_{6}+z_{3}^2+z_{4}^2\right)\,.
\]
As the random variables $Y(\phi)$ and $Z(\phi)$
are related by
\[
Y(\phi)=\tH[Z(\phi)]\,,
\]
it then follows from~\eqref{rhoHZ} that the density
of $Y(\phi)$ is given by the marginal distribution
\begin{equation*}
	 \rho_{Y(\phi)}(z'')=\int_{\bR^4} \rho_{H[Z(\phi)]}(z)\, dz'\,.
       \end{equation*}

       Now notice that the function $Q(z)$ defined in~\eqref{defQz} is simply
       \[
Q(z)=|H^{-1}(z)|^2\,.
         \]
As the matrix $\Si$ is positive definite, therefore there is a positive
constant $c>0$
such that
\[
\rho_{Y(\phi)}(z'')\lesssim\int_{\bR^4}\frac{e^{-cQ(z)}}{z_2^2}\, dz'\,.
\]
Lemma~\ref{L.intQ} then ensures that
$\sup_{|z''|<\frac12}\rho_{Y(\phi)}(z'')\lesssim1$. Lemma~\ref{L.Bulin}
then guarantees that the random function~$Y$ does not vanish on~$\TT$
almost surely, and the theorem follows.
\end{proof}

Theorem~\ref{T.as} is then proven, modulo the proof of
Lemma~\ref{L.intQ}, which we will address next.

\subsection{Proof of the main technical lemma}

Let us now present the proof of Lemma~\ref{L.intQ}. To make the
exposition clearer, we will divide the proof in three steps.

\subsubsection{The integral~$\tI$}

The first step is to rewrite the integral
\begin{equation*}
I:=\int_{\RR^4} \frac{e^{-c Q(z)}}{z_2^2}dz'
\end{equation*}
in a more convenient way. For this, let us set
  \[
\vr:=z_1z_3-z_5\,,\qquad \tau:=\frac{z_1z_3-z_5}{z_2}\,.
  \]
  The map $z'\mapsto (\vr,\tau,z_3,z_4)$ is invertible outside the
  hyperplane $z_3=0$ and the set $\tau=0$. In terms of these variables, the integral reads
  as
  \[
I=\int_{\RR^4}\frac{e^{-c \,Q_1}}{|\vr z_3|}\,d\vr\,d\tau\, dz_3\,dz_4
  \]
  with
  \begin{align}
Q_1&:= Q_2 + z_4^2\bigg[1+\bigg( \frac{\tau (\vr+z_5)}{\vr
     z_3}\bigg)^2\bigg]+ 2z_4 (\tau^2+z_3^2-z_6)\frac{\tau^2 (\vr+z_5)}{\vr^2
     z_3}\,,\notag\\[1mm]
    Q_2&:= z_3^2+\tau^2 +\frac{\vr^2}{\tau^2}+ \bigg(\frac{\tau(\tau^2+z_3^2-z_6)}{\vr}\bigg)^2+\bigg(\frac{\vr+z_5}{z_3}\bigg)^2\,.\label{Q2}
  \end{align}
As $Q_1$ is a second order polynomial in~$z_4$, one can explicitly
integrate in this variable, obtaining
\[
I(z'')= \sqrt{\frac{\pi}{c}}
\int_{\RR^3}\frac{e^{-cQ_3}}{\sqrt{\vr^2z_3^2+\tau^2(\vr+z_5)^2}}\,
d\vr\, d\tau\, dz_3\,,
\]
with
\[
Q_3:= z_3^2+\tau^2 +\frac{\vr^2}{\tau^2}+ \bigg(\frac{\tau z_3(\tau^2+z_3^2-z_6)}{(z_3^2\vr^2+\tau^2(\vr+z_5)^2)^{1/2}}\bigg)^2+\bigg(\frac{\vr+z_5}{z_3}\bigg)^2\,.
\]

Let us now consider polar coordinates $(\si,\al)\in \RR^+\times\TT$, defined as
\[
z_3=:\si\cos\al\,,\qquad \tau=:\si\sin\al\,.
\]
Still denoting by $Q_2$ the expression of~\eqref{Q2} in these
variables, and similarly with the other functions~$Q_j$, we get
\[
Q_2=\frac{\vr^2}{\si^2}\csc^2\al+\Big(\frac{\vr+z_5}{\si}\Big)^2\sec^2\al+\si^2+ \bigg(\frac{\si(\si^2-z_6)\sin\al}\vr\bigg)^2\,.
\]
This enables us to write
\begin{align*}
I&= \sqrt{\frac{\pi}{c}}\int_{-\infty}^\infty
          \int_0^{2\pi}\int_0^\infty \frac{e^{-cQ_3}}{\sqrt{\vr^2
          \cos^2\al + (\vr +z_5)^2\sin^2\al}}\, d\si\, d\al\, d\vr\,.
\end{align*}

As $|z''|<\frac12$, the denominator is nonzero for $|\vr|>1$, so one obviously has
\[
\int_{\RR\backslash[-1,1]}
          \int_0^{2\pi}\int_0^\infty \frac{e^{-cQ_3}}{\sqrt{\vr^2
          \cos^2\al + (\vr +z_5)^2\sin^2\al}}\, d\si\, d\al\,
      d\vr\lesssim \int_{\RR}
          \int_0^\infty e^{-c(\si^2+\frac{\vr^2}{\si^2})}\, d\si\,d\vr\lesssim1\,.
\]
We can then write
\begin{equation}\label{tI}
I\lesssim 1+\int_{-1}^1
          \int_0^{2\pi}\int_0^\infty \frac{e^{-cQ_3}}{\sqrt{\vr^2
              \cos^2\al + (\vr +z_5)^2\sin^2\al}}\, d\si\, d\al\, d\vr=:1+ \tI\,.
        \end{equation}

        \subsubsection{The case $z_5=0$}
Let us start by assuming that $z_5=0$, so that
        \begin{align*}
\tI&= \int_{-1}^1
          \int_0^{2\pi}\int_0^\infty \frac{e^{-cQ_3}}{|\vr|}\, d\si\,
     d\al\, d\vr\leq2\int_0^1 \int_0^{2\pi}\int_{-\infty}^\infty \frac{e^{-c\si^2 -
            c\vr^{-2} \si^2(\si^2-z_6)^2\sin^2\al\cos^2\al}}{\vr}\, d\si\,
     d\al\, d\vr\,.
        \end{align*}
        The integral in~$\vr$ can be computed in terms of the
        incomplete Gamma function
        \[
\Ga(\la,x):=\int_x^\infty t^{\la-1} e^{-t}\,dt\,,
\]
obtaining
\[
\tI\leq  \int_0^{2\pi}\int_{-\infty}^\infty e^{-c\si^2}\Ga[0,c \si^2(\si^2-z_6)^2\sin^2\al\cos^2\al]\,d\si\,
     d\al\,.
\]
Then the bound
\[
\Ga(0,x)\lesssim \log\bigg(2+\frac1x\bigg)\,,
\]
valid for all $x>0$,
immediately implies that
\begin{equation}\label{x50}
\sup_{|z_6|<\frac12}\tI\lesssim \int_0^{2\pi}\int_{-\infty}^\infty
e^{-c\si^2} \log\bigg(2+\frac1{c \si^2(\si^2-1/2)^2\sin^2\al\cos^2\al}\bigg) \,d\si\,
     d\al\lesssim 1
   \end{equation}
   when $z_5=0$.

\subsubsection{The case $z_5\neq0$}

In view of the estimate~\eqref{x50}, from now on, we shall assume that
$z_5\neq0$. Let us now define the new variable $\tvr:=-\vr/z_5$, in
terms of which the integral~$\tI$ reads as
\[
\tI\leq\int_{-1/|z_5|}^{1/|z_5|}
          \int_0^{2\pi}\int_0^\infty \frac{e^{-cQ_4}}{S(\tvr,\al)}\, d\si\, d\al\, d\tvr\,.
        \]
        Here we have used that
        \[
\sqrt{\vr^2
              \cos^2\al + (\vr +z_5)^2\sin^2\al}=|z_5|\, S(\tvr,\al)
        \]
        with
        \[
S(\tvr,\al):=\sqrt{\tvr^2
              \cos^2\al + (\tvr -1)^2\sin^2\al}
        \]
        and $Q_4$ is defined as
\[
Q_4:= \si^2+\frac{\si^2(\si^2-z_6)^2}{z_5^2S(\tvr,\alpha)^2}\sin^2\al\cos^2\alpha.
  \]

Let us fix some small~$\ep>0$ and define the sets
\[
\cM_0:=\{(\tvr,\al): |\tvr|<\ep\,,\; |\sin\al|<\ep\}\,,\qquad \cM_1:=\{(\tvr,\al): |\tvr-1|<\ep\,,\; |\cos\al|<\ep\}\,.
\]
Since $S(\tvr,\al)\gtrsim 1$ for $(\tvr,\al)\not\in \cM_0\cup\cM_1$ (not uniformly in $\ep$), let
us consider the set
\[
\cM_2:=\left(\left(-\frac1{|x_5|},\frac1{|x_5|}\right)\times \TT\right)
\backslash (\cM_0\cup\cM_1)
\]
and split the above integral as
\begin{align*}
  \tI= \int_{\cM_0}\int_0^\infty + \int_{\cM_1}\int_0^\infty +\int_{\cM_2}\int_0^\infty
  =:\tI_0+\tI_1+\tI_2\,.
\end{align*}

To estimate~$\tI_0$, observe that $\cM_0$ consists of two connected
components, which are contained in $|\tvr|<\ep$ and either $
|\al|<C\ep$ or $ |\al-\pi|<C\ep$, respectively. It is easy to
see that both contributions to the integral are of the same size, so
we will just consider the first. To analyze it, let us use the bound
\[
S(\tvr,\al)\gtrsim \sqrt{\tvr^2+\al^2}\,,
\]
which clearly holds for $(\tvr,\al)\in\cM_0^+$, to write
\begin{align*}
\tI_0&\lesssim \int_{-\ep}^{\ep}\int_{-C\ep}^{C\ep}
          \int_0^\infty \frac{e^{-cQ_4}}{S(\tvr,\al)}\,
       d\si\, d\al\, d\tvr\\
  &\lesssim \int_{-\ep}^{\ep}\int_{-C\ep}^{C\ep}
          \int_0^\infty \frac{e^{-c\si^2}}{\sqrt{\tvr^2+\al^2}}\,
       d\si\, d\al\, d\tvr\,.
\end{align*}
Once can now introduce a new set of polar coordinates
\[
\tvr =: r\cos\be\,,\qquad \al=: r\sin\be\,,
\]
which yields
\begin{align*}
  \tI_0&\lesssim \int_{0}^{C\ep}
          \int_0^{2\pi}\int_0^\infty e^{-c\si^2}\,d\si\,d\be\,dr\lesssim1\,.
\end{align*}
An analogous argument for~$\cM_1$, where $|\tvr-1|<\ep$ and either
$|\al-\frac\pi2|<C\ep$ or $|\al-\frac{3\pi}2|<C\ep$, shows that
\[
\tI_1\lesssim 1\,.
\]

It only remains to bound~$\tI_2$. As $S(\tvr,\al)\gtrsim \langle\tvr\rangle$
on~$\cM_2$, where
$\langle x\rangle:=(1+x^2)^{\frac12}$ is the Japanese bracket, we can write
\begin{align*}
\tI_2&\lesssim \int_{-1/|z_5|}^{1/|z_5|}
          \int_0^{2\pi}\int_0^\infty \frac1{\tvr} e^{-c
       \si^2-c\frac{\si^2(\si^2-z_6)^2}{z_5^2S(\tvr,\alpha)^2}\sin^2\al\cos^2\alpha}\,
       d\si\, d\al\, d\tvr\\
  &= 4\int_{-1/|z_5|}^{1/|z_5|}
          \int_0^{\pi/2}\int_0^\infty\frac1{\tvr} e^{-c
       \si^2-c\frac{\si^2(\si^2-z_6)^2}{z_5^2S(\tvr,\alpha)^2}\sin^2\al\cos^2\alpha}\, d\si\, d\al\, d\tvr\,.
\end{align*}
As $\cos^2\al\sin^2\al=\frac14\sin^2(2\al)$ and $\sin\al\gtrsim \al$  for $|\al|<\frac\pi2$, the integral
in~$\al$ can be estimated as
\begin{align*}
&\int_0^{\pi/2}e^{-c\frac{\si^2(\si^2-z_6)^2}{z_5^2{S}(\tvr,\alpha)^2}\sin^2\al\cos^2\alpha}\,d\al\leq
                                                                           \int_0^{\pi/2}e^{-C
                                                                           \frac{\si^2(\si^2-z_6)^2}{z_5^2\tilde{S}(\tvr)^2}\sin^2(2\al)}\,d\al\\
  & =2\int_0^{\pi/4}e^{-C
  	\frac{\si^2(\si^2-z_6)^2}{z_5^2\tilde{S}(\tvr)^2}\sin^2(2\al)}\,d\al\lesssim \bigg\langle\frac{\si(\si^2-z_6)}{z_5\tilde{S}(\tvr)}\bigg\rangle^{-1}\,,
\end{align*}
where $\tilde{S}(\tvr)\coloneqq \tvr^2+(1-\tvr)^2$. Here we have used that for $c>0$
\[
\int_0^{\pi/4} e^{-c^2x^2}dx=\frac{\sqrt{\pi } \Erf\left(\frac{\pi  c}{4}\right)}{2 {c}}\lesssim \langle {c}\rangle^{-1} \,,
\]
where $\Erf$ is the error function. Since $|z_6|\leq \frac12$, this yields
\begin{align*}
  \tI_2&\lesssim \int_{-1/|z_5|}^{1/|z_5|}
          \int_0^\infty \frac{e^{-c
       \si^2}}{\langle\tvr\rangle} \bigg\langle\frac{\si(\si^2-z_6)}{z_5 \tilde{S}(\tvr)}\bigg\rangle^{-1}\, d\si\,d\tvr\\
         &=\int_{-1}^1\int_0^\infty\frac{e^{-c
       \si^2}}{(z_5^2+\vr^2)^\frac12} \frac{|z_5\tilde{S}(\tvr)|}{(\vr^2+(\vr+z_5)^2+\si^2(\si^2-z_6)^2)^{1/2}}\, d\si\,d\vr\\
       &\leq \int_{-1}^1\int_0^\infty\frac{e^{-c
       \si^2}}{(\vr^2+(\vr+z_5)^2+\si^2(\si^2-1/2)^2)^{1/2}}\, d\si\,d\vr\,.
\end{align*}
where we have used that if $z_5=a\rho$
\begin{equation*}
	\frac{\left(z_5 \tilde{S}(\tvr)\right)^2}{\rho ^2+z_5^2}=\frac{\rho ^2+(\rho +z_5)^2}{\rho ^2+z_5^2}=\frac{a^2+2 a+2}{a^2+1}<C
\end{equation*}
for some $C>0$ and for all $a\in \bR$.
To integrate in $\vr$, we need that
\begin{multline*}
	\int_{-1}^1 \frac{1}{\sqrt{\left((a+\rho )^2+\rho ^2\right)+b}} \, d\rho\\
	=\frac{1}{\sqrt{2}}\log \left(\frac{\left(\sqrt{2} \sqrt{(a-2) a+b+2}-a+2\right) \left(\sqrt{2} \sqrt{a (a+2)+b+2}+a+2\right)}{a^2+2 b}\right)\,.
\end{multline*}
Using that $|z_5|<\frac12$ we conclude
\begin{equation*}
	\tI_2\lesssim \int_0^\infty e^{-c\si^2}	\log\left(\frac{\left(\sqrt{2} \sqrt{4 \sigma ^6-4 \sigma ^4+\sigma ^2+13}+5\right)^2}{2 \sigma ^2 \left(1-2 \sigma ^2\right)^2}\right)d\si
\end{equation*}
Thus, we obtain the bound
\[
\tI_2\lesssim 1\,,
\]
from the fact that the logarithmic singularities at $\sigma=0$ and $\sigma= 1/{\sqrt{2}}$ are integrable. Lemma~\ref{L.intQ} in then proven.

\section*{Acknowledgements}

A.E.\ is supported by the ERC Consolidator Grant~862342. D.P.-S.
is supported by the grants MTM PID2019-106715GB-C21 (MICINN) and Europa Excelencia EUR2019-103821 (MCIU). A.R.\ is supported by the grant MTM PID2019-106715GB-C21 (MICINN). This work is supported in part by the ICMAT--Severo Ochoa grant CEX2019-000904-S. A.R. is also a posgraduate fellow of the City Council of Madrid at the Residencia de Estudiantes (2020-2021).

\appendix
\section{Monochromatic waves with many nondegenerate critical points}
\label{A.manycp}

In this Appendix we aim to prove that there exist solutions to the
Helmholtz equation
\begin{equation*}
\De v + v=0
\end{equation*}
on the plane with many isolated critical points. Specifically, let
\[
N^*(\nabla v,R):=\left\{x\in B_R:\nabla v(x)=0\,,\; \det\nabla^2 v(x)\neq0\right\}
\]
be the number of nondegenerate critical points of~$v$ contained in the
ball of radius~$R$. One can then prove the following:

\begin{proposition}\label{P.manycp}
Given any continuous function $\rho:\RR^2:\to\RR^+$, there exists a
solution to the Helmholtz equation on~$\RR^2$ such that
\[
N^*(\nabla v,R)> \rho(R)
\]
for all $R>1$.
\end{proposition}

\begin{proof}
Without any loss of generality, let us assume that the function~$\rho$
is increasing. Take a set of distinct points $\{x_k\}_{k\in\NN}\subset \mathbb R^2$ without any
accumulation points such that
\begin{equation}\label{Eq.AppA1}
	\#\{k\in\NN: x_k\in B_R\}>\rho(R+\tfrac12)
\end{equation}
for all $R>\frac18$. At each point~$x_k$, consider the number
\[
r_k:= \frac18\min\left\{ 1, \inf_{j\in\NN\backslash\{ k\}}{|x_k-x_j|}\right\}\,,
\]
which is positive because the set $\{x_k\}_{k\in\NN}$ does not have
any accumulation points.

The function $v_k(x):= J_0(|x-x_k|)$ satisfies the Helmholtz
equation on the plane and $x_k$ is a nondegenerate maximum
of~$v_k$ (in fact, $D^2 v_k(x_k)=-\frac12I$). Therefore, the implicit function theorem ensures that there
exists some $\ep_k>0$ such that any function~$v$ with~$\|v_k-v\|_{C^2(B(x_k,2r_k))}<\ep_k$ has a nondegenerate local
maximum inside the ball $B(x_k,r_k)$. Notice that $B(x_k,2r_k)\cap B(x_j,2r_j)=\emptyset$ if $k\neq j$.

The better-than-uniform global approximation theorem for the
Helmholtz equation~\cite[Lemma~7.2]{EPS13} ensures that there exists a
solution~$v$ to the Helmholtz equation on~$\RR^2$ such that
\[
\sup_{k\in\NN} \frac{\|v_k-v\|_{C^2(B(x_k,2r_k))}}{\ep_k}<1\,.
  \]
One then infers that $v$ has a nondegenerate critical point in each
disk $B(x_k,r_k)$. The property~\eqref{Eq.AppA1} then ensures that $N^*(\nabla v,R)> \rho(R)
$ for all $R>1$, as claimed.
\end{proof}

\begin{remark}
The result and the proof remain valid in higher dimensions. The only
modification is that, on~$\RR^n$, one must define $v_k(x):=
|x-x_k|^{1-\frac n2}J_{\frac{n}{2}-1}(|x-x_k|)$.
      \end{remark}

      \begin{remark}
The function~$v$ may not be polynomially bounded at infinity,
so $v$ does not need to have a Fourier transform. In particular, it
does not need to be the Fourier transform of a distribution
supported on the unit sphere.
\end{remark}

\section{The translation-invariant
  case}\label{ApCompTI}

In this Appendix we shall see why the evaluation of the Kac--Rice
integral that gives the asymptotic behavior of $\bE\crit$
(cf.~Lemma~\ref{L:KR est}) is so much
easier in the translation-invariant case (that is, when $s=0$ following Remark \ref{KernelNS}).

In the translation-invariant case, it is easy to work directly in
Cartesian coordinates, instead of using polar coordinates. This is
because all one needs to know about~$u$
in order to apply the Kac--Rice formula are expectation
values of the form $\bE[\pd^\al u(x)\,\pd^\be u(x)]$, where $\al,\be$
are multiindices of order at most~2. These quantities can be computed
exactly using that, as discussed in Remark~\ref{KernelNS}, for $s=0$
the covariance kernel is (up to a normalizing constant)
\begin{equation}\label{addition1}
K(x,x')=J_0(|x-x'|)= \int_{\TT}e^{i\xi\cdot (x-x')}\, d\si(\xi)\,.
\end{equation}
Indeed, taking derivatives in this
expression one finds that
\[
\bE[\pd^\al u(x)\,\pd^\be u(x)]=i^{|\al|-|\be|}\int_{\TT}
\xi^\al\,\xi^\be \, d\si(\xi)\,.
\]
The last integral can be computed in closed form because~\cite{Fol01}
\[
\int_\TT \xi^\al\, d\si(\xi)= \begin{cases}
{\pi^{-1}}\big[\prod_{j=1}^{2}\Gamma(\frac{\al_j+1}2)\big] /\Gamma(\frac{|\al|+2}2) & \text{if $\al_1,\al_2$
	are even,}\\[1mm]
0 & \text{otherwise.}
\end{cases}
\]
These formulas readily show that $\bE[\pd_ju\, \pd_{kl}u]=0$, so
$\nabla u$ and $\nabla^2 u$ are independent Gaussian random functions,
and that the covariance matrices of the first and second derivatives
of~$u$ are
\begin{equation*}
\var \nabla u(x)={\frac{1}{2}} I\,,\qquad \var\nabla^2 u(x)=\frac{1}{8}  \left(
\begin{array}{ccc}
3 & 0 & 1 \\
0 & 1 & 0 \\
1 & 0 & 3 \\
\end{array}
\right)\,.
\end{equation*}
Again, we have regarded $\nabla^2 u$ as a 3-component vector. By
the Kac--Rice formula, these
expressions are enough to show
\begin{equation}\label{ka0easy}
\bE N(\nabla u, R)=\pi R^2\int_{\RR^3}\frac{\left| z_1^2+2 \sqrt{2}
    z_1 z_2-z_2^2\right| }{8
  \pi}\frac{e^{-\frac12|z|^2}}{(2\pi)^{3/2}}dz=\ka(0) R^2\,
\end{equation}
as in Remark \ref{rem:kappa 0}.

In polar coordinates, one sees essentially the same
simplifications. The point is that it suffices to differentiate the addition formula
\[
g(r,r',\te):=J_0\big(\sqrt{r^2+r'^2-2rr'\cos\te}\big)=
\sum_{l=0}^\infty{\epsilon_l} J_l(r) J_l(r')\cos l\te\,,
\]
where $\epsilon_l:=2-\de_{l,0}$ is  Neumann's factor,
to compute in closed form all the sums
appearing in the Kac--Rice formula (Lemma~\ref{L:KR est}). Incidentally, the addition formula is equivalent to the assertion that the covariance matrix of~$u$
is~\eqref{addition1}, written in polar coordinates. For example,
\begin{align*}
  \sum_{l=0}^\infty{\epsilon_l} J_l(r)^2&=g(r,r,0)=1\,,\\
  \sum_{l=0}^\infty{\epsilon_l} J_l'(r)^2&=\pd_r\pd_{r'}g(r,r,0)=\frac12\,,\\
  \sum_{l=0}^\infty{\epsilon_l} l^2J_l(r)J_l'(r)&=-\frac12\pd_r\pd^2_\te
                                      g(r,r,0)=\frac r4\,,\\
    \sum_{l=0}^\infty{\epsilon_l} l^4J_l(r)^2&=\pd_\te^4g(r,r,0)=\frac{r^2(4+3r^2)}8\,.
\end{align*}
These formulas are exact and easy to obtain, as one does not need to
carry out the hard frequency analysis that constitutes the core of
this paper. Of course, one can plug the values of these sums in
Lemma~\ref{L:KR est} to readily recover the formula~\eqref{ka0easy} for
the expected number of critical points.

\bibliographystyle{amsplain}

\end{document}